\documentclass[11pt, a4paper, twoside,reqno]{amsart}

\usepackage[utf8]{inputenc}
\usepackage[british]{babel}
\usepackage[T1]{fontenc}
\usepackage{amsmath,amsfonts,amsthm,latexsym,amssymb}
\usepackage{microtype}
\usepackage{fancyhdr} 
\usepackage{lmodern}
\usepackage{xcolor}
\usepackage[a4paper]{geometry}

\usepackage{tikz-cd}
\usepackage{nccmath}
\usepackage{enumerate}
\usepackage{empheq}
\usepackage[bookmarks=true,breaklinks=true,bookmarksnumbered = true,colorlinks=false,hidelinks]{hyperref}
\usepackage{theoremref}
\usepackage{bm}

\usepackage{faktor}
\allowdisplaybreaks

\usepackage{graphicx}
\usepackage{pdflscape}
\usepackage{epstopdf}

\theoremstyle{plain}
\newtheorem{thm}{Theorem}[section]
\newtheorem{thmA}{Theorem}

\newtheorem{prop}[thm]{Proposition}
\newtheorem{cor}[thm]{Corollary}
\newtheorem{lem}[thm]{Lemma}

\theoremstyle{remark}
\newtheorem{rem}[thm]{Remark}
\newtheorem{notaz}[thm]{Notation}

\theoremstyle{definition}
\newtheorem{defin}[thm]{Definition}
\newtheorem{exam}[thm]{Example}

\newcommand\Z{\mathbb{Z}}

\newcommand\CC{\mathcal{C}}

\DeclareMathOperator{\Id}{Id}
\DeclareMathOperator{\Ker}{Ker}

\newcommand\od{\overline{d}}

\newcommand{\myrightleftarrows}[1]{\mathrel{\substack{\xrightarrow{#1} \\[-.9ex] \xleftarrow{#1}}}}

\makeatletter

\renewcommand{\p@enumii}{}

\def\@enum@{\list{\csname label\@enumctr\endcsname}%
	{\usecounter{\@enumctr}\def\makelabel##1{
			\normalfont\ignorespaces\emph{{##1}~}}
		\setlength{\labelsep}{3pt}
		\setlength{\parsep}{0pt}
		\setlength{\itemsep}{0pt}
		\setlength{\leftmargin}{0pt}
		\setlength{\labelwidth}{0pt}
		\setlength{\listparindent}{\parindent}
		\setlength{\itemsep}{0pt}
		\setlength{\itemindent}{0pt}
		\topsep=3pt plus 1pt minus 1 pt}}
\makeatother

\makeatletter
\@namedef{subjclassname@2020}{%
	\textup{2020} Mathematics Subject Classification}
\makeatother

\begin{document}

	\title[Cactus and twin groups]{Cactus groups, twin groups, and right-angled Artin groups}
	
	\author[Paolo Bellingeri]{Paolo Bellingeri}
	\address{Normandie Univ, UNICAEN, CNRS, LMNO, 14000 Caen, France}
	\email{paolo.bellingeri@unicaen.fr}
	
	\author[Hugo Chemin]{Hugo Chemin}
	\address{Normandie Univ, UNICAEN, CNRS, LMNO, 14000 Caen, France}
	\email{hugo.chemin@unicaen.fr}
	
	\author[Victoria Lebed]{Victoria Lebed}
	\address{Normandie Univ, UNICAEN, CNRS, LMNO, 14000 Caen, France}
	\email{victoria.lebed@unicaen.fr}
	
	\date{\today}
	
	%\dedicatory{(In progress)}
	
	
	%conjugation problem, 
	
	\begin{abstract}
		Cactus groups $J_n$ are currently attracting considerable interest from diverse mathematical communities. This work explores their relations to right-angled Coxeter groups, and in particular twin groups $Tw_n$ and Mostovoy's Gauss diagram groups $D_n$, which are better understood. 
		Concretely, we construct an injective group $1$-cocycle from $J_n$ to $D_n$, and show that $Tw_n$ (and its $k$-leaf generalisations) inject into $J_n$. As a corollary, we solve the word problem for cactus groups, determine their torsion (which is only even) and center (which is trivial), and answer the same questions for pure cactus groups, $PJ_n$. In addition, we yield a $1$-relator presentation of the first non-abelian pure cactus group $PJ_4$. Our tools come mainly from combinatorial group theory.
		%OLD Cactus groups $J_n$ have recently gained considerable interest from diverse mathematical communities. This work explores several maps relating them to groups which are better understood: right-angled Artin and Coxeter groups (RAAGs and RACGs), and in particular twin groups $Tw_n$. Concretely, we construct an injective group $1$-cocycle from $J_n$ to a RACG (the Gauss diagram group $D_n$ of Mostovoy); show that $Tw_n$ (and its $k$-leaf generalisations) inject into $J_n$; and prove that any RAAG/RACG injects into $Tw_{\infty}$. As a corollary, we solve the word and conjugacy problems for $J_n$, determine their torsion (which is only even) and center (which is trivial), and answer the same questions for the pure cactus groups $PJ_n$. In addition, we yield a $1$-relator presentation of the first non-trivial pure cactus group $PJ_4$.
	\end{abstract}

	\maketitle
	%%%%%%%%%
	
	\section{Introduction}\label{intro}
	
	Cactus groups appeared under the name of \emph{quasibraid groups} in the study of the \emph{mosaic operad}; this latter governs the moduli space of configurations of smooth points on punctured stable real algebraic curves of genus zero \cite{Devadoss_MosaicOp,Etingof_StableCurves,KhoWill_Revisited}. They were immediately generalised to other Coxeter types, and renamed \emph{mock reflection groups} \cite{Davis_BlowUps}.
	
	It was later realised that the same groups control \emph{coboundary categories}, just as braid groups control braided categories \cite{Henriques_CoboundCat}. That paper launched the term \emph{cactus groups}, inspired by the Opuntia-cactus-like form of the moduli spaces above. Coboundary categories were designed to study the crystals of finite-dimensional reductive Lie algebras and, more generally, the representations of coboundary Hopf algebras. 
	
	Cactus groups also appear in the context of hives and octahedron recurrence \cite{KTW_OctaRecurrence,Henriques_Octahedron}. Together with their generalisations to other Coxeter types, they have become a recurrent tool in representation theory \cite{Bonnafe_CellsCacti,Losev_CactiCells,Chmutov_BK_Cactus}.

	Concretely, the \emph{cactus group} $J_n$ is defined by its generators\footnote{We should have written $s_{p,q;n}$ here. However, we systematically drop the subscript $n$ since it is always clear from the context. The same is done for the maps $s$ and $d$, and for the Gauss diagrams $\tau_I$ below.}  $s_{p,q}$, where $1\le p < q \le n$, and relations
	\begin{align}
	&s_{p,q}^2=1, \label{E:j1}\tag{j1}\\
	&s_{p,q} s_{m,r}=s_{m,r}s_{p,q} \; \text{ if }\ [p,q] \cap [m,r] = \emptyset, \label{E:j2}\tag{j2}\\
	&s_{p,q} s_{m,r}=s_{p+q-r,p+q-m}s_{p,q} \; \text{ if }\ [m,r]  \subset [p,q]. \label{E:j3}\tag{j3}
	\end{align}
	%OLD This definition was extended in \cite{Bonnafe_CellsCacti,Yu_LinGenCactus} to all Coxeter groups, and in \cite{Losev_CactiCells} to the Weyl group of a Lie algebra. 
	
	The generator $s_{p,q}$ can be diagrammatically represented as the braid on $n$ strands where the strands $p$, $p+1$, $\ldots$, $q$ intersect at one common point, and reverse their order after that point. The relations are depicted in  Fig.~\ref{Pic:cactus_relations}. Here and below the diagrams are drawn from left to right, in order to match the order of generators in a word representing a cactus.
	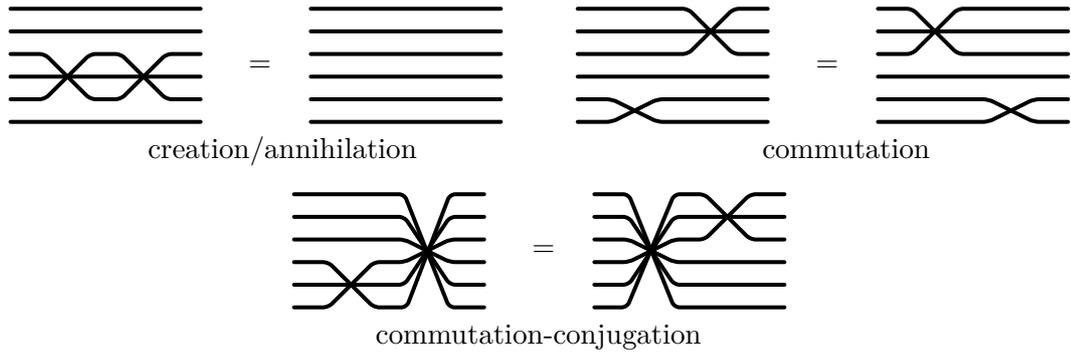
\begin{figure}[h]
		\begin{center}
		\begin{tikzpicture}[line cap=round,line join=round,x=.5cm,y=.3cm,rounded corners=2pt,line width=1.5pt]
\draw (0.5,0.)-- (5.5,0.);
\draw (0.5,4.)-- (5.5,4.);
\draw (0.5,5.)-- (5.5,5.);
\draw (0.5,3.)-- (1.4,3.)-- (2.6,1.)-- (3.4,1.)-- (4.6,3.)-- (5.5,3.);
\draw (0.5,1.)-- (1.4,1.)-- (2.6,3.)-- (3.4,3.)-- (4.6,1.)-- (5.5,1.);
\draw (0.5,2.)-- (5.5,2.);
\node at (7.3,2.5) {$=$ \hspace*{3pt}};		
\end{tikzpicture}
\begin{tikzpicture}[line cap=round,line join=round,x=.5cm,y=.3cm,rounded corners=2pt,line width=1.5pt]
\draw (0.5,0.)-- (5.5,0.);
\draw (0.5,4.)-- (5.5,4.);
\draw (0.5,5.)-- (5.5,5.);
\draw (0.5,3.)-- (5.5,3.);
\draw (0.5,1.)-- (5.5,1.);
\draw (0.5,2.)-- (5.5,2.);		
\end{tikzpicture}	
\hspace*{20pt}
\begin{tikzpicture}[line cap=round,line join=round,x=.5cm,y=.3cm,rounded corners=2pt,line width=1.5pt]
\draw   (0.5,1.)-- (1.4,1.)-- (2.6,0.)-- (5.5,0.);
\draw   (0.5,0.)-- (1.4,0.)-- (2.6,1.)-- (5.5,1.);
\draw   (0.5,2.)-- (5.5,2.);
\draw   (0.5,3.)-- (3.4,3.)-- (4.6,5.)-- (5.5,5.);
\draw   (0.5,5.)-- (3.4,5.)-- (4.6,3.)-- (5.5,3.);
\draw   (0.5,4.)-- (3.4,4.)-- (4.6,4.)-- (5.5,4.);
\node at (7.3,2.5) {$=$ \hspace*{3pt}};	
\end{tikzpicture}	
\begin{tikzpicture}[line cap=round,line join=round,x=.5cm,y=.3cm,rounded corners=2pt,line width=1.5pt]
\draw   (0.5,1.)-- (3.4,1.)-- (4.6,0.)-- (5.5,0.);
\draw   (0.5,0.)-- (3.4,0.)-- (4.6,1.)-- (5.5,1.);
\draw   (0.5,2.)-- (5.5,2.);
\draw   (0.5,3.)-- (1.4,3.)-- (2.6,5.)-- (5.5,5.);
\draw   (0.5,5.)-- (1.4,5.)-- (2.6,3.)-- (5.5,3.);
\draw   (0.5,4.)-- (1.4,4.)-- (2.6,4.)-- (5.5,4.);
\end{tikzpicture} 
creation/annihilation \hspace*{0.25\textwidth} commutation

\bigskip
\begin{tikzpicture}[line cap=round,line join=round,x=.5cm,y=.3cm,rounded corners=2pt,line width=1.5pt]
\draw   (0.5,0.)-- (1.4,0.)-- (2.6,2.)-- (3.4,2.)-- (4.6,3.)-- (5.5,3.);
\draw   (0.5,2.)-- (1.4,2.)-- (2.6,0.)-- (3.46,0.)-- (4.6,5.)-- (5.5,5.);
\draw   (0.5,1.)-- (1.4,1.)-- (2.6,1.)-- (3.4,1.)-- (4.6,4.)-- (5.5,4.);
\draw   (0.5,3.)-- (3.4,3.)-- (4.6,2.)-- (5.5,2.);
\draw   (0.5,4.)-- (3.4,4.)-- (4.6,1.)-- (5.5,1.);
\draw   (0.5,5.)-- (3.4,5.)-- (4.58,0.)-- (5.5,0.);
\node at (7.3,2.5) {$=$ \hspace*{3pt}};	
\end{tikzpicture}
\begin{tikzpicture}[line cap=round,line join=round,x=.5cm,y=.3cm,rounded corners=2pt,line width=1.5pt]
\draw   (0.5,5.)-- (1.4,5.)-- (2.6,0.)-- (5.5,0.);
\draw   (0.5,4.)-- (1.4,4.)-- (2.6,1.)-- (5.5,1.);
\draw   (0.5,3.)-- (1.4,3.)-- (2.6,2.)-- (5.5,2.);
\draw   (0.5,2.)-- (1.4,2.)-- (2.6,3.)-- (3.4,3.)-- (4.6,5.)-- (5.5,5.);
\draw   (0.5,1.)-- (1.4,1.)-- (2.6,4.)-- (3.4,4.)-- (4.6,4.)-- (5.5,4.);
\draw   (0.5,0.)-- (1.36,0.)-- (2.6,5.)-- (3.4,5.)-- (4.6,3.)-- (5.5,3.);
\end{tikzpicture}

commutation-conjugation
			\caption{Three types of relations in cactus groups}\label{Pic:cactus_relations}
		\end{center}
	\end{figure}
	These diagrams make one think of cacti once again---saguaros these time. For this reason we will often call \emph{cacti} the elements of~$J_n$, and use the term \emph{leaf number} for the parameter $q-p+1$ of the generator $s_{p,q}$.
	
	One should handle such diagrams with care: the braid relation $s_{1,2}s_{2,3}s_{1,2}=s_{2,3}s_{1,2}s_{2,3}$ from Fig.~\ref{Pic:braid}, natural in braid and knot theories (where it corresponds to the Reidemeister III move), does not hold in cactus groups.
	\begin{figure}[h]
		\begin{center}
    \begin{tikzpicture}[line cap=round,line join=round,x=.5cm,y=.3cm,rounded corners=2pt,line width=1.5pt]
\draw   (0.,0.)-- (3.36,0.)-- (4.6,1.)-- (5.4,1.)-- (6.6,2.)-- (8.,2.);
\draw   (0.,1.)-- (1.4,1.)-- (2.6,2.)-- (5.4,2.)-- (6.6,1.)-- (8.,1.);
\draw   (0.,2.)-- (1.4,2.)-- (2.6,1.)-- (3.4,1.)-- (4.64,0.)-- (8.,0.);
\node at (9.5,1) {$\neq$ \hspace*{3pt}};
\end{tikzpicture}			
	\begin{tikzpicture}[line cap=round,line join=round,x=.5cm,y=.3cm,rounded corners=2pt,line width=1.5pt]
\draw   (0.,2.)-- (3.4,2.)-- (4.6,1.)-- (5.4,1.)-- (6.56,0.)-- (8.,0.);
\draw   (0.,1.)-- (1.4,1.)-- (2.6,0.)-- (5.4,0.)-- (6.6,1.)-- (8.,1.);
\draw   (0.,0.)-- (1.4,0.)-- (2.6,1.)-- (3.4,1.)-- (4.6,2.)-- (8.,2.);
\node at (9.5,1) {$\neq$ \hspace*{3pt}};	
\end{tikzpicture}
    \begin{tikzpicture}[line cap=round,line join=round,x=.5cm,y=.3cm,rounded corners=2pt,line width=1.5pt]
\draw   (0.,2.)-- (1.4,2.)-- (2.6,0.)-- (4.,0.);
\draw   (0.,1.)-- (1.4,1.)-- (2.6,1.)-- (4.,1.);
\draw   (0.,0.)-- (1.32,0.)-- (2.6,2.)-- (4.,2.);
%\node at (5.5,1) {$\neq$ \hspace*{3pt}};	
\end{tikzpicture}
			\caption{The braid relation, false in cactus groups}\label{Pic:braid}
		\end{center}
	\end{figure}
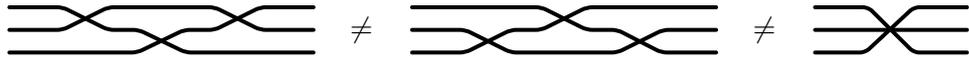
	
	The closure of such braids yields \emph{cactus doodles}, i.e. curves with self-intersections \cite{Mostovoy_CactusDoodles}.
	
	Looking at how such braids permute their strands, one obtains a group morphism
	\begin{align*}
	s \colon J_n &\to S_n,\\
	s_{p,q} &\mapsto (1,2,\ldots,p-1,\ \bm{q,q-1,\ldots,p+1,p},\ q+1,q+2,\ldots,n).
	\end{align*}
	The kernel of this map is the \emph{pure cactus group} $PJ_n$, sometimes denoted as $\Gamma_{n+1}$. It is the fundamental group of the real locus of the Deligne--Mumford
	compactiﬁcation $\overline{\mathcal{M}}_{0,n+1}$ of the moduli space of rational curves with $n+1$ marked points \cite{Devadoss_MosaicOp}. This explains why these groups are particularly interesting.
	
	Our braid-like diagrams can be read in another way: label the strands from $1$ (top) to $n$ (bottom), and then at each multiple point write down the set of the labels of the intersecting strands. This yields a set-theoretic map
	\[d \colon J_n \to D_n.\]
	Here $D_n$ is the \emph{Gauss diagram group} from \cite{Mostovoy_CactResNilpot}\footnote{In \cite{Mostovoy_CactResNilpot}, the $D_n$ were simply called \emph{diagram groups}. Following a suggestion of Mostovoy, we use a more precise term, in order to avoid confusion with Guba and Sapir's diagram groups.}. Concretely, it has one generator $\tau_I$ for each subset $I$ of $\{1,2,\ldots,n\}$ of size $\geq 2$, and the relations are
	\begin{align}
	&\tau_I^2=1, \label{E:d1}\tag{d1}\\
	&\tau_I \tau_J = \tau_J \tau_I \; \text{ if }\ I \cap J = \emptyset \ \text{ or }\ I \subset J. \label{E:d2}\tag{d2}
	\end{align}
	It is a \emph{right-angled Coxeter group} (\emph{RACG}), that is, it is generated by idempotents with only commutation relations between them. Its elements will be called \emph{Gauss diagrams}, since they are related to the Gauss diagrams from virtual knot theory.
	
	The map $d$ is not a group morphism. In Section~\ref{S:WordPb}, we explain that it is a \emph{group $1$-cocycle}, injective by a theorem from \cite{Mostovoy_CactResNilpot} (see also \cite{Yu_LinGenCactus} for a proof for other Coxeter types). However, its restriction to the pure part $PJ_n$ becomes a group morphism. 
	
	The ``reading'' maps $s$ and $d$ can be assembled into a single injective group morphism 
	\[\rho = d \times s \colon J_n \to D_n \rtimes S_n.\]
	The semi-direct product on the right can be seen as the \emph{virtual cactus group}, where any, not necessarily neighboring, collection of strands can come together (using the $S_n$ part) and form a multi-strand intersection (using the $D_n$ part). Note that, contrary to the usual approach to virtuality in similar settings  \cite{Bardakov_TwinStructureVirtual,Nanda_VirtTw}, in $D_n \rtimes S_n$ a diagram $\tau_{I} \in D_n$ and a permutation $\sigma \in S_n$ do commute when $\sigma$ permutes elements from~$I$ only.
	
	In Section~\ref{S:WordPb}, we use the map~$d$ to reduce the \emph{word problem} in $J_n$ to its much easier analogue in the RACG $D_n$, and describe an efficient solution.

	%In Section~\ref{S:WordPb}, we use the map~$d$ to reduce the \emph{word and conjugacy problems} in $J_n$ to their much easier analogues in the RACG $D_n$.
	
	Similarly, in Section~\ref{S:Subgroups} we work on the $D_n$ side to study certain subgroups of~$J_n$. Concretely, given some $2 \leq i \leq j \leq n$, consider the group $J_n^{i,j}$ defined by the generators $s_{p,q}$, where $1\le p < q \le n$ and $i \leq q-p+1 \leq j$, and the cactus relations \eqref{E:j1}-\eqref{E:j3}. In other words, we keep only those generators whose leaf number is between $i$ and $j$. The groups $Tw_n=J_n^{2,2}$ appeared as \emph{Grothendieck cartographical groups} in \cite{Voevodsky_cartographical}; further as \emph{twin groups} in \cite{Khovanov_DoodleGroups}, as a diagrammatic description of the motion of $n$ points on the plane without triple collisions, but also as a tool to study \emph{doodles} (closed plane curves without triple intersections, see \cite{Fenn_doodles}); later under the name of \emph{flat braids} \cite{Merkov_flat} and \emph{planar braids} \cite{Mostovoy_PureTwin6}; and finally under the name of \emph{traids} in physics literature \cite{traid_groups}. By definition, the twin groups are RACGs, just like any group $J_n^{i,i}$. As other RACG families, they appear in several contexts such as topological robotics \cite{GLM_robotics}. Our first results is
	\begin{thmA}
		For all $2 \leq i \leq j \leq n$, the natural maps
		\begin{align*}
		&J_n^{i,j} \to J_n,\\
		&s_{p,q} \mapsto s_{p,q}
		\end{align*}
		are injective.
	\end{thmA}
	Thus cactus groups contain twin groups and their higher-leaf analogues. This result is actually established for a wider class of partial presentations of $J_n$. It can also be seen as the braid-like counterpart of the recent proof that the space of doodles embeds into that of cactus doodles \cite{Mostovoy_CactusDoodles}.
	
	We further exploit the injectiveness of~$d$ in Section~\ref{S:Torsion} to study the \emph{torsion} and the \emph{center} of $J_n$ and $PJ_n$. We prove
	\begin{thmA}
		The cactus groups $J_n$ have no odd torsion. Moreover, for any $k$ they contain torsion of order $2^k$ provided that $n$ is big enough.
	\end{thmA}
	%Even better: we explicitely describe all torsion elements in $J_n$.
	\begin{thmA}
		The pure cactus groups $PJ_n$ are torsionless.
	\end{thmA}
	\begin{thmA}
		The groups $J_n$ and $PJ_n$ are centerless whenever $n>2$ and $n>3$ respectively.
	\end{thmA}
	
	We have seen that the cactus group $J_n$ contains several important RACGs ($Tw_n$ and more generally all the $J_n^{i,i}$). In the opposite direction, it injects (non-homomorphically) into the RACG $D_n$. In some sense, it can be thought of as a deformation of a RACG, where some commutation relations are deformed to commutation-conjugation relations. In fact, it can be seen as the Coxeter-like finite quotient, in the sense of \cite{LV_StrGroups}, of the structure group of a partial solution to the Yang--Baxter equation, in the sense of \cite{Chouraqui_Partial}. It inherits some properties of the RACG~$D_n$, but loses others: for instance, it has less center (the center of~$D_n$ is $\langle\tau_{1,2,\ldots,n}\rangle \simeq \Z_2$) and more torsion ($D_n$ has torsion of order $2$ but not of order $4$).  
	
	Finally, in Appendix~\ref{A:PJ4} we yield a (complicated) one-relator presentation for the first interesting pure cactus group~$PJ_4$. In particular, it confirms the absence of torsion in this group.
	
	Some of our results can be recovered using methods from topological algebra or geometric group theory. Thus, one can explain the absence of torsion in $PJ_n$ by interpreting it as the fundamental group of an aspherical manifold \cite{Etingof_StableCurves}. 
	Another approach, pointed to us by Anthony Genevois, exploits the median property of the natural Cayley graph of $J_n$. This property implies that cactus groups are  $\operatorname{CAT}(0)$ groups,
	and therefore (see for intance \cite{BriHae}) have solvable word and conjugacy problems. 
	Our proofs, of combinatorial nature, have the advantage of being elementary, explicit, and self-contained.
	
	\medskip
	\textbf{Acknowledgements.} The authors are grateful to Neha Nanda and John Guaschi for their help with GAP computations and fruitful conversations, and to Jacob Mostovoy and Anthony Genevois for helpful discussions and remarks. P.B. was partially supported by the ANR project AlMaRe (ANR-19-CE40-0001).
	%hhh reviewer + ANR
	% In Section~\ref{S:RACG}, we show that the connections between cactus groups an RACGs are even richer:  
	%\begin{thmA}
	%Any RACG injects into the twin group $Tw_n$ (and hence into the cactus groups $J_n$), provided that $n$ is big enough. The same is true for any RAAG.
	%\end{thmA}
	%Recall that a \emph{right-angled Artin group} (\emph{RAAG}) is given by a set of generators with only commutation relations between some of them. Its quotient by the squares of all the generators is a RACG, just like the symmetric group $S_n$ is a quotient of the Artin braid group $B_n$.

	\section{Word problem for cactus groups}\label{S:WordPb}
	% and conjugacy problems
	
	Recall that, given a cactus diagram $t$ representing a cactus $c \in J_n$, the Gauss diagram $d(c) \in D_n$ is constructed as follows: label the left endpoints of the strands of $t$ from $1$ (top) to $n$ (bottom); at each crossing reverse the order of the strands, and hence of the labels; at the $i$th crossing write down the ()unordered) set $I_i$ of the labels of the intersecting strands; finally, multiply the generators of $D_n$ corresponding to these label sets from left to right, setting $d(c)=\tau_{I_1}\tau_{I_2}\cdots$. See Fig.~\ref{Pic:cactus_example} for an example\footnote{Here and below we write $\tau_{a,b,\ldots}$ instead of $\tau_{ \{a,b,\ldots\} }$ for simplicity.}.
	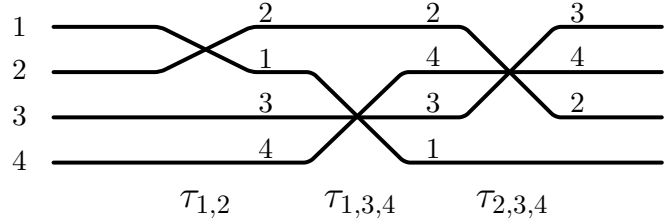
\begin{figure}[h]
		\begin{center}
					\begin{tikzpicture}[line cap=round,line join=round,x=1cm,y=.6cm,rounded corners=2pt,line width=1.5pt]
\draw (0.,3.)-- (1.4,3.)-- (2.6,2.)-- (3.4,2.)-- (4.62,0.)-- (8.,0.);
\draw (0.,2.)-- (1.4,2.)-- (2.6,3.)-- (5.4,3.)-- (6.6,1.)-- (8.,1.);
\draw (0.,1.)-- (5.4,1.)-- (6.6,3.)-- (8.,3.);
\draw (0.,0.)-- (3.36,0.)-- (4.6,2.)-- (8.,2.);
\draw (1.5,-0.36) node[anchor=north west] {\Large$\tau_{1,2}$};
\draw (3.4,-0.36) node[anchor=north west] {\Large$\tau_{1,3,4}$};
\draw (5.4,-0.36) node[anchor=north west] {\Large$\tau_{2,3,4}$};
\draw (-0.7,3.46) node[anchor=north west] {$1$};
\draw (2.54,2.76) node[anchor=north west] {$1$};
\draw (4.74,0.76) node[anchor=north west] {$1$};
\draw (-0.7,2.5) node[anchor=north west] {$2$};
\draw (2.54,3.76) node[anchor=north west] {$2$};
\draw (6.64,1.78) node[anchor=north west] {$2$};
\draw (-0.7,1.5) node[anchor=north west] {$3$};
\draw (4.74,1.78) node[anchor=north west] {$3$};
\draw (6.64,3.76) node[anchor=north west] {$3$};
\draw (-0.7,0.5) node[anchor=north west] {$4$};
\draw (4.74,2.8) node[anchor=north west] {$4$};
\draw (6.64,2.8) node[anchor=north west] {$4$};
\draw (2.54,0.76) node[anchor=north west] {$4$};
\draw (2.54,1.78) node[anchor=north west] {$3$};
\draw (4.74,3.76) node[anchor=north west] {$2$};
			\end{tikzpicture}
			\caption{A cactus $c$ with $s(c)=(4312)$ and $d(c)=\tau_{1,2}\tau_{1,3,4}\tau_{2,3,4}$}\label{Pic:cactus_example}
		\end{center}
	\end{figure}
	
	\begin{thm}[\cite{Mostovoy_CactResNilpot,Yu_LinGenCactus}]
		The above procedure yields a well-defined injective map 
		\[d \colon J_n  \hookrightarrow D_n.\]
	\end{thm}
	
	This is a consequence of the following two lemmas, which we will also need below.
	
	\begin{notaz}
		Let $FJ_n$ (resp., $FD_n$) be the free group on the generators $s_{p,q}$ (resp., $\tau_I$).	
	\end{notaz}
	
	The above procedure defines a map 
	\[\od \colon FJ_n \hookrightarrow FD_n,\] 
	which is clearly injective, but not surjective for $n> 2$ (for instance, $\tau_{1,3}$ is not in its image). A careful comparison of the relations defining $J_n$ and $D_n$ yields
	
	\begin{lem}\label{L:d_lift}
		\begin{enumerate}
			\item If a word $w' \in FJ_n$ is obtained from $w \in FJ_n$ by applying a relation of type \eqref{E:j1} (resp., \eqref{E:j2} or \eqref{E:j3}), then $\od(w') \in FD_n$ is obtained from $\od(w) \in FD_n$ by applying a relation of type \eqref{E:d1} (resp., \eqref{E:d2}).
			\item\label{it:pullback} Conversely, if a word $v' \in FD_n$ is obtained from some $\od(w) \in FD_n$ by applying an {annihilation relation} $\tau_I^2 \leadsto 1$ (resp., a commutation relation \eqref{E:d2}), then $v'=d(w')$ for the word  $w' \in FJ_n$ obtained from $w$ by applying a corresponding relation of type \eqref{E:j1} (resp., \eqref{E:j2} or \eqref{E:j3}).
		\end{enumerate}
	\end{lem}
	In other words, both commutation and commutation-conjugation relations for cacti are translated by commutation relations for Gauss diagrams. 
	
	Note that the statement~\eqref{it:pullback} is false for {creation relations} $1 \leadsto \tau_I^2$, since these latter can lead outside of the image of~$\od$.
	
	The following result is standard in the theory of RACGs:
	\begin{lem}\label{L:RACG_normal}
		Let $G$ be a RACG, and $w$ a word in the standard generators (called \emph{letters}) representing an element $g \in G$. Consider the following procedure: as long as $w$ contains two copies of the same letter $l$ separated by letters commuting with $l$, move one copy towards the other by commutation, then annihilate them by applying $l^2 \leadsto 1$; repeat. The result of this procedure is independent, up to commutation in~$G$, of the choice of the annihilated couples and of the choice of the word $w$ representing $g$.
	\end{lem}
	
	In particular, one can transform any two words representing the same element of a RACG into the same word without ever applying the creation relation.
	
	These two lemmas immediately imply that $\od$ induces an injective map $d \colon J_n \to D_n$.
	
	In fact, the second lemma yields more. Choose any order on the set of generators of a RACG $G$, and extend it lexicographically to the words in these generators. Since by Lemma~\ref{L:RACG_normal} all minimal-length representatives of an element $g$ of a RACG $G$ are equivalent up to commutation, the minimal word among such representatives yields an easily computable \emph{normal form} on~$G$. According to Lemma~\ref{L:d_lift}, this normal form can be pulled back from $D_n$ to $J_n$. This gives a solution to the \emph{word problem} in $J_n$, which we summarise as follows:
	
	\begin{prop}\label{P:cactus_normal}
		Let $w \in FJ_n$ be a word representing a cactus $c \in J_n$. Consider the following procedure: if $w$ contains two letters $l$ and $l'$ such that $l'$ can be (conjugation-)commuted all the way to~$l$ (according to the rules \eqref{E:j2}-\eqref{E:j3}) and in the process becomes $l$, then do this (conjugation-)commutation and annihilate $ll$ (according to the rule \eqref{E:j1});  repeat as long as possible. The result is the empty word if and only if the cactus $c$ is trivial.
	\end{prop}
	
	This procedure has a nice diagrammatic interpretation if one works with $i$-leaf cacti only (that is, with elements from $J_n^{i,i}$). It then consists in \emph{bigon killing}, as illustrated in Fig.~\ref{Pic:bigon}.
	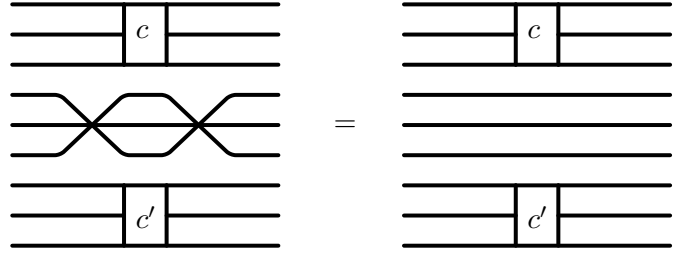
\begin{figure}[h]
		\begin{center}
		\begin{tikzpicture}[line cap=round,line join=round,x=.7cm,y=.4cm,rounded corners=2pt,line width=1.5pt]
\draw   (0.5,2.)-- (1.4,2.)-- (2.6,4.)-- (3.4,4.)-- (4.6,2.)-- (5.5,2.);
\draw   (0.5,3.)-- (5.5,3.);
\draw   (0.5,4.)-- (1.4,4.)-- (2.6,2.)-- (3.4,2.)-- (4.6,4.)-- (5.5,4.);
\draw   (2.6,7.)-- (3.4,7.);
\draw   (3.4,7.)-- (3.4,5.);
\draw   (3.4,5.)-- (2.6,5.);
\draw   (2.6,5.)-- (2.6,7.);
\draw   (2.6,1.)-- (3.4,1.);
\draw   (3.4,1.)-- (3.4,-1);
\draw   (2.6,-1)-- (2.6,1.);
\draw   (0.5,1.)-- (2.6,1.);
\draw   (0.5,0.)-- (2.6,0.);
\draw   (3.4,1.)-- (5.5,1.);
\draw   (3.4,0.)-- (5.5,0.);
\draw   (0.5,-1)-- (5.5,-1);
\draw   (0.5,5.)-- (2.6,5.);
\draw   (3.4,5.)-- (5.5,5.);
\draw   (0.5,6.)-- (2.6,6.);
\draw   (3.4,6.)-- (5.5,6.);
\draw   (0.5,7.)-- (2.6,7.);
\draw   (3.4,7.)-- (5.5,7.);
\draw (2.6,6.7) node[anchor=north west] {$c$};
\draw (2.6,0.7) node[anchor=north west] {$c'$};
\node at (7,3) {$=$ \hspace*{6pt}};	
\end{tikzpicture}
		\begin{tikzpicture}[line cap=round,line join=round,x=.7cm,y=.4cm,rounded corners=2pt,line width=1.5pt]
\draw   (0.5,2.)-- (5.5,2.);
\draw   (0.5,3.)-- (5.5,3.);
\draw   (0.5,4.)-- (5.5,4.);
\draw   (2.6,7.)-- (3.4,7.);
\draw   (3.4,7.)-- (3.4,5.);
\draw   (3.4,5.)-- (2.6,5.);
\draw   (2.6,5.)-- (2.6,7.);
\draw   (2.6,1.)-- (3.4,1.);
\draw   (3.4,1.)-- (3.4,-1);
\draw   (2.6,-1)-- (2.6,1.);
\draw   (0.5,1.)-- (2.6,1.);
\draw   (0.5,0.)-- (2.6,0.);
\draw   (3.4,1.)-- (5.5,1.);
\draw   (3.4,0.)-- (5.5,0.);
\draw   (0.5,-1)-- (5.5,-1);
\draw   (0.5,5.)-- (2.6,5.);
\draw   (3.4,5.)-- (5.5,5.);
\draw   (0.5,6.)-- (2.6,6.);
\draw   (3.4,6.)-- (5.5,6.);
\draw   (0.5,7.)-- (2.6,7.);
\draw   (3.4,7.)-- (5.5,7.);
\draw (2.6,6.7) node[anchor=north west] {$c$};
\draw (2.6,0.7) node[anchor=north west] {$c'$};
\end{tikzpicture}
			\caption{Bigon killing; here $c$ is any cactus from $J_a^{3,3}$, and $c'$ is any cactus from $J_{n-a-3}^{3,3}$}\label{Pic:bigon}
		\end{center}
	\end{figure}
	
	\begin{defin}
		A word $w \in FJ_n$ is called \emph{irreducible} if it contains no two letters that can be (conjugation-)commuted together and annihilated.
	\end{defin}
	
	Pulling back form~$D_n$ to~$J_n$ the results of Lemma~\ref{L:RACG_normal}, one sees that all irreducible representatives of a cactus $c\in J_n$ are related by (conjugation-)commutation. In particular, they have the same length, which is minimal for representatives of $c$.
		
\begin{rem}
The conjugacy problem in~$J_n$ is much more delicate. In particular, conjugation may shorten even very simple irreducible words. The word $s_{3,4}s_{1,2}s_{1,4}s_{3,6} \in J_6$ illustrates this phenomenon:
\begin{align*}
s_{5,6}s_{3,4} \cdot (s_{3,4}s_{1,2}s_{1,4}s_{3,6}) \cdot s_{3,4}s_{5,6} &= s_{5,6} s_{1,2}s_{1,4}s_{3,6}  s_{3,4}s_{5,6} = s_{5,6} s_{1,4}s_{3,4}s_{3,6}  s_{5,6}s_{3,4}  \\
&= s_{5,6} s_{1,4}s_{3,6}s_{5,6}  s_{5,6}s_{3,4} = s_{1,4} s_{5,6} s_{3,6} s_{3,4} \\
&= s_{1,4} s_{3,6} s_{3,4} s_{3,4} = s_{1,4} s_{3,6}.
\end{align*}	
\end{rem}
	
	We finish this section with a remark on the nature of the map $d$. It is not a group morphism: for example, applied to the cactus from Fig.~\ref{Pic:cactus_example}, it yields $\tau_{1,2}\tau_{1,3,4}\tau_{2,3,4}$, whereas a group morphism would have given $\tau_{1,2}\tau_{2,3,4}\tau_{1,2,3}$. However, it is not so far from being one. It is in fact a \emph{group $1$-cocycle}, that is, it satisfies the twisted compatibility relation
	\[d(c_1c_2) = d(c_1) \ {}^{c_1}\!d(c_2), \qquad c_1,c_2 \in J_n,\]
	where the left group action of $J_n$ on $D_n$ is induced from the label-permuting $S_n$-action on $D_n$: ${}^{c}\! t =  {}^{s(c)}\! t$, with $c \in J_n$, $t \in D_n$. In the example from Fig.~\ref{Pic:cactus_example}, we obtain
	\[d(s_{1,2}s_{2,4}s_{1,3})=\tau_{1,2} \ {}^{(2134)}\!\tau_{2,3,4} \ {}^{(4132)}\!\tau_{1,2,3}=\tau_{1,2}\tau_{1,3,4}\tau_{2,3,4}.\]
	
	Note that, restricted to the pure part $PJ_n$, the map $d$ becomes a group morphism, since $s(c)=\Id$ for a pure cactus $c$.

	\section{Twin groups are subgroups of cactus groups}\label{S:Subgroups}
	
	Consider a group $G = \langle S \ | \ R \rangle$ defined by a set of generators $S$ and a set of relations $R$. For any subset of generators $I \subseteq S$, one can extract from~$R$ all the relations $R_I$ involving the generators from $I$ only. This defines a new group $G_I:=  \langle I \ | \ R_I \rangle$, with the obvious map
	\begin{align*}
	\iota_I \colon G_I &\to G,\\
	g &\mapsto g \quad \text{ for all } g \in I.
	\end{align*} 
	Such maps will be called \emph{$\iota$-type maps} in what follows. They need not be injective. 
	\begin{defin}
		A subset of generators $I$ is called \emph{complete} if the above map $\iota_I$ is injective.
	\end{defin}
	
	\begin{exam}
		In a RACG or a RAAG with its standard presentation, any generator subset is complete. It follows from Lemma~\ref{L:RACG_normal} and its analogue for RAAGs.
	\end{exam}
	\begin{exam}
		All generator subsets are complete in braid and symmetric groups with their standard presentations as well.
	\end{exam}
	This is actually true for more general Artin--Tits and Coxeter groups. %hhh ref ???
	
	\begin{exam}
		In the virtual braid group $G=VB_3$ with its classical generators $\sigma_1,\sigma_2$ and virtual generators $\tau_1,\tau_2$ and the usual relations, the set $I=\{\sigma_1,\tau_1,\tau_2\}$ is not complete. Indeed, since there are no relations relating $\sigma_1$ only to the $\tau$'s, $G_I$ is the direct product $G_I = \Z * S_3$, and it includes two distinct elements 
		\[\sigma_1\tau_1\tau_2\sigma_1\tau_2\tau_1\sigma_1 \ \text{ and } \  \tau_1\tau_2\sigma_1\tau_2\tau_1\sigma_1\tau_1\tau_2\sigma_1\tau_2\tau_1,\]
		sent by $\iota_I$ to the same element $\sigma_1\sigma_2\sigma_1 = \sigma_2\sigma_1\sigma_2$ of $G$, since $\sigma_2=\tau_1\tau_2\sigma_1\tau_2\tau_1$ in $G$.
	\end{exam}
	
	Cactus groups provide some more interesting counterexamples:
	\begin{exam}
		In the cactus group $G=J_4$ with its usual presentation, there are no relations involving only $s_{1,2}$ and $s_{1,4}$, except for idempotence relations. Thus, for $I=\{s_{1,2},s_{1,4}\}$, one gets $G_I = \Z_2 \ *\ \Z_2$. However, in $G$ these generators satisfy the relation
		\[s_{1,2}s_{1,4}s_{1,2}s_{1,4} = s_{1,4}s_{1,2}s_{1,4}s_{1,2},\]
		which expresses the commutation of $s_{1,2}$ and $s_{3,4} =s_{1,4}s_{1,2}s_{1,4}$. 	
	\end{exam}	
	In the example above, one should add the generator $s_{3,4} =s_{1,4}s_{1,2}s_{1,4}$ to make the set $I=\{s_{1,2},s_{1,4}\}$ complete. We will now prove that this is the only possible completeness defect in cactus groups. More precisely, a generator subset is complete al long as it is stable by certain conjugations.
	
	\begin{defin}
		A collection $\CC$ of sub-intervals of the integer interval $[1,n]$ is called \emph{symmetric} if together with any two nested sub-intervals $[m,r]  \subset [p,q]$ it contains the sub-interval $[p+q-r,p+q-m]$, symmetric to $[m,r]$ with respect to the middle of $[p,q]$.
	\end{defin}
	
	\begin{thm}\label{T:subgroup}
		For any symmetric collection $\CC$ of sub-intervals of $[1,n]$, the family $\{s_I | I \in \CC\}$ of generators of the cactus group $J_n$ (with its standard presentation) is complete.
	\end{thm}
	
	Before giving a proof, let us describe several important particular cases.
	
	\begin{cor}\label{C:CactCact}
		The group $J_n$ can be viewed as a subgroup of $J_{n+k}$, via the map $s_{p,q} \mapsto s_{p,q}$.
	\end{cor}
	
	\begin{cor}\label{C:TwCact}
		The twin group $Tw_n$ can be viewed as a subgroup of the cactus group $J_n$, via the map $s_{p,p+1} \mapsto s_{p,p+1}$.
	\end{cor}
	
	More generally, given some $2 \leq i \leq j \leq n$, the sub-interval collection
	\[\CC_{i,j}:=\{[p,q] \ | \ 1\le p < q \le n,\ i \leq q-p+1 \leq j\}\]
	is clearly symmetric. Theorem~\ref{T:subgroup} thus applies to the group $J_n^{i,j}$ defined by the generators $s_{p,q}$, where $[p,q] \in \CC_{i,j}$, and the cactus relations \eqref{E:j1}-\eqref{E:j3}. In other words, in $J_n^{i,j}$ we keep only those generators whose leaf number is between $i$ and $j$. We get
	
	\begin{cor}\label{C:SliceCact}
		The group $J_n^{i,j}$ can be viewed as a subgroup of $J_n$, via the map $s_{p,q} \mapsto s_{p,q}$.
	\end{cor}
	
	\begin{proof}[Proof of Theorem~\ref{T:subgroup}]
		Take a symmetric collection $\CC$ of sub-intervals of $[1,n]$. Consider a word $w \in FJ_n$ which contains only generators  $s_I$ with $I \in \CC$, and which represents the trivial element in $J_n$. We need to show that it also represents the trivial element in $(J_n)_I$. According to Proposition~\ref{P:cactus_normal}, the word $w$ can be turned into the trivial word by applying commutation, commutation-conjugation and annihilation relations. But all these relation are also available in the group $(J_n)_I$; in fact, the symmetry condition on~$I$ was imposed precisely to preserve all commutation-conjugation relations from $J_n$ in $(J_n)_I$.
	\end{proof}
	
	\begin{rem}
		If one is interested in the $i$-leaf group $J_n^{i,i}$ only (for instance, the twin group $Tw_n=J_n^{2,2}$), then in the arguments above the Gauss diagram group $D_n$ can be replaced with a smaller RACG. Concretely, consider the \emph{width $i$ Gauss diagram group} $D_n^i$ generated by the idempotents $\tau_I$ for all $i$-element subsets $I$ of $\{1,2,\ldots,n\}$, which commute if the corresponding subsets are disjoint. The symmetric group~$S_n$ still acts on such subsets $I$, and hence on $D_n^i$. Consider the following \emph{eraser map}:
		
		\begin{align*}
		\varepsilon_i \colon J_n &\to D_n^i \rtimes S_n,&\\
		s_{p,q} &\mapsto 1 &\text{ if } q-p+1 <i,\\
		s_{p,q} &\mapsto (\tau_{[p,q]},s(s_{p,q})) &\text{ if } q-p+1 =i,\\
		s_{p,q} &\mapsto (1,s(s_{p,q})) &\text{ if } q-p+1 >i.
		\end{align*}
		Going through the defining relations \eqref{E:j1}-\eqref{E:j3} of~$J_n$, one checks that this map is well defined. Now, in the diagram below, the rectangle and the square clearly commute:
		\[ \begin{tikzcd}
		J_n^i \arrow{r}{\iota} &J_n \arrow{r}{\rho} &D_n \rtimes S_n \arrow{r}{\pi_1} &D_n\\
		J_n^i \arrow{r}{\iota} \arrow{u}{\Id} &J_n \arrow{r}{\varepsilon_i} &D_n^i \rtimes S_n \arrow{u}{\iota} \arrow{r}{\pi_1} &D_n^i \arrow{u}{\iota}.
		\end{tikzcd}
		\]
		Here we abusively use the same notation $\iota$ for all $\iota$-type maps, and the same notation $\pi_1$ for all (set-theoretic) projections onto the first component of a semi-direct product. Then the injectivity of the total map of the first line implies the injectivity for the second line. In other words, we obtain an injective group $1$-cocycle $J_n^i \to D_n^i$.
		
		\noindent In the same vein, the symmetric group~$S_n$ above can be replaced with the subgroup $S_n^i$ generated by all the size $i$ flops $s(s_{p,q})$. It would be interesting to understand the structure of these permutation subgroups.
	\end{rem}
	
	The eraser map from the above remark admits the following variation:
	\begin{align*}
	\epsilon_i \colon J_n &\to J_n^{i,n},&\\
	s_{p,q} &\mapsto 1 &\text{ if } q-p+1 <i,\\
	s_{p,q} &\mapsto s_{p,q} &\text{ if } q-p+1 \geq i.
	\end{align*}
	In other words, it erases all generators with leaf number $<i$. A quick direct verification shows that it is well defined and surjective; the map $\iota_i \colon J_n^{i,n} \to J_n$, $s_{p,q} \mapsto s_{p,q}$, is its section (cf. Corollary~\ref{C:SliceCact}). The subgroup $J_n^{2,i-1} \overset{\iota'_i}{\hookrightarrow} J_n$, and hence its normal closure $\langle\langle J_n^{2,i-1} \rangle\rangle$, is by construction in the kernel of $\epsilon_i$. We will now prove that this is the whole kernel. In particular, this yields the following semi-direct decompositions of the cactus groups:
	\[J_n \simeq \langle\langle J_n^{2,i-1} \rangle\rangle \rtimes J_n^{i,n}.\]
	
	\begin{prop}\label{P:CactusSlice}
		The maps above define the following split exact sequence:
		\[0 \longrightarrow \langle\langle J_n^{2,i-1} \rangle\rangle \overset{\iota'_i}{\longrightarrow} J_n \underset{\iota_i}{\overset{\epsilon_i}{\myrightleftarrows{\rule{.7cm}{0cm}}}} J_n^{i,n} \longrightarrow  0.\]
	\end{prop}
	
	\begin{proof}
		It remains to prove that any cactus $c$ in the kernel of the eraser map $\epsilon_i$ lies in fact in the normal closure $\langle\langle J_n^{2,i-1} \rangle\rangle$. Take a word $w \in FJ_n$ representing~$c$. Since $c \in \Ker(\epsilon_i)$, one can erase all the letters from~$w$ with leaf number $<i$, and then (permutation-)commute together and annihilate pairs of remaining letters, in well-chosen order, until the word becomes empty, as explained in Proposition~\ref{P:cactus_normal}. Now, this (permutation-)commutation and annihilation can still be performed when the  ``small'' letters are not erased: to move a letter $l$ over a small letter $m$ (or its conjugate), simply replace $m$ by its $l$-conjugate, since $lm = (lml)l$ and $ml=l(lml)$. When the process stops, one is left with a product of conjugates of small letters representing~$c$. 
	\end{proof}
	
	One can push the above arguments slightly further and show that, for any $2 \leq i' \leq i \leq j \leq j' \leq n$, $J_n^{i,j}$ can be viewed as a subgroup of $J_n^{i',j'}$, via the map $s_{p,q} \mapsto s_{p,q}$. This defines a functor from the poset of integer sub-intervals of $[1,n]$ to the category of subgroups of~$J_n$. Moreover, for any $2 \leq i' \leq i \leq j  \leq n$, one has the decomposition
	\[J_n^{i',j} \simeq \langle\langle J_n^{i',i-1} \rangle\rangle \rtimes J_n^{i,j}.\]
	
	A possible application of these constructions is the filtration 
	\[\langle\langle J_n^{2,n-2} \rangle\rangle \triangleleft J_n^{2,n-1} \triangleleft J_n\]
	with RACG quotients
	\begin{align*}
	J_n / J_n^{2,n-1} \simeq J_n^{n,n} &\simeq \Z_2,\\
	J_n^{2,n-1} / \langle\langle J_n^{2,n-2} \rangle\rangle  \simeq J_n^{n-1,n-1} &\simeq \Z_2 \ast \Z_2 \; \text{ if } n \geq 3.
	\end{align*}
	However, understanding the structure of the next piece, $\langle\langle J_n^{2,n-2} \rangle\rangle$, seems difficult even for $n=4$.
	
\begin{rem}
In this section, we have seen that $J_n$ contains many RACG subgroups. It would be interesting to find out whether \textbf{all} RACGs can be realised inside cactus groups. For instance, a tedious direct verification shows that one can include any RACG with $\leq 5$ generators into a (sufficiently big) cactus group by sending each generator to a generator (as usual for the standard presentation), except for the ``pentagon'' group 
\[\langle g_1, \ldots,g_5 \ | \ \forall i, g_i^2=1 \text{ and } g_i g_{i+1} = g_{i+1}g_i \rangle.\]
Here $g_6$ is identified with $g_1$.	
\end{rem}	
	
	\section{Torsion and center of cactus groups}\label{S:Torsion}
	
	Many basic group-theoretic questions are easy to answer for a RACG $G$. For instance,
	\begin{enumerate}[(1)]
		\item Its center $Z(G)$ is generated by all its \emph{friendly} generators (that is,  the generators of~$G$ commuting with all other generators). Thus $Z(G) \simeq \Z_2^f$, $f$ being the number of the friendly generators of~$G$.
		\item The only torsion $G$ has is of order $2$. More precisely, $2$-torsion elements are the conjugates of products of pairwise commuting generators.
	\end{enumerate} 
	In particular, for the Gauss diagram group $D_n$, we have the center
	\[Z(D_n) = \langle \tau_{1,2,\ldots,n} \rangle \simeq \Z_2,\]
	and a big $2$-torsion part, without any other torsion.
	
	The aim of this section is to determine the center and the torsion of the cactus group $J_n$ and its pure part $PJ_n$. Our main tool is the connection between $J_n$ and the RACG $D_n$. Curiously, the answers are close to but different from those for $D_n$.
	
	\begin{thm}\label{T:Center}
		The cactus group $J_n$ is centerless whenever $n>2$.
	\end{thm}
	
	In the case $n=2$, we have $J_2 \simeq \Z_2$, and $PJ_2$ is trivial. We will no longer mention this case in what follows.
	
	\begin{proof}
		Let $w \in FJ_n$ be a word representing a non-trivial central element $c \in J_n$. It can be assumed to be of minimal length among representatives of non-trivial central elements. We will show that the Gauss diagram $d(c)$ is then central in $D_n$. As recalled above, this would imply $d(c)= \tau_{1,2,\ldots,n}$ or $1$. Since $d$ is injective and $c$ non-trivial, this means $c=s_{1,n}$. But the generator $s_{1,n}$ does not commute with $s_{1,2}$ when $n>2$, since
		\[d(s_{1,n}s_{1,2})=\tau_{1,2,\ldots,n}\tau_{n,n-1} \neq \tau_{1,2,\ldots,n} \tau_{1,2} = \tau_{1,2} \tau_{1,2,\ldots,n} = d(s_{1,2}s_{1,n}).\]
		Thus there are no non-trivial central elements in $J_n$.
		
		Take a generator $s_{p,q}$ of~$J_n$. Since $c s_{p,q} = s_{p,q} c$, Lemmas \ref{L:d_lift} and \ref{L:RACG_normal} leave us with two options.
		
		\textbf{Option 1:} The words $ws_{p,q}$ and $s_{p,q}w$ are irreducible. Then $s_{p,q}w$ can be transformed to $ws_{p,q}$ by commutation and commutation-conjugation relations only. In particular, a letter $l$ in $s_{p,q}w$ can be (conjugation-)commuted to the end of the word and yield the letter $s_{p,q}$.
		
		\begin{description}
			\item[Case 1] The letter $l$ is the initial letter $s_{p,q}$ of $s_{p,q}w$. At the level of Gauss diagrams, this means that all letters in the word $\od(w) \in FD_n$ commute with ${}^{s(c)}\!\tau_{p,\ldots,q}$ in $D_n$.
			\item[Case 2]\label{C:Case2} The letter $l$ is from the word $w$. This means that (conjugation-)commutation can transform $w$ into $w's_{p,q}$. But then $ws_{p,q} = w's_{p,q}s_{p,q} = w'$ in~$J_n$, and the word $ws_{p,q}$ is no longer minimal.
		\end{description}
		
		\textbf{Option 2:} The words $ws_{p,q}$ and $s_{p,q}w$ are reducible (simultaneously, since all minimal representatives of a cactus have the same length). Recalling that the word~$w$ is minimal, and looking what this means for $ws_{p,q}$ on the $D_n$ side, one concludes that we are in the situation of the Case 2 above: a letter $l'$ can be (conjugation-)commuted to the end of~$w$, so that $w$ becomes $w's_{p,q}$. But the same argument applied to $s_{p,q}w$ shows that a letter $l''$ can be (conjugation-)commuted to the beginning of~$w$, so that $w$ becomes $s_{p,q}w''$. Again, two cases are possible.
		
		\begin{description}
			\item[Case 1] The letters $l'$ and $l''$ occupy the same position in~$w$. At the level of Gauss diagrams, this means that all letters in the word $\od(w) \in FD_n$ commute with ${}^{s(c)}\!\tau_{p,\ldots,q}$ in $D_n$.
			\item[Case 2] The letters $l'$ and $l''$ occupy different positions in~$w$. Then (conjugation-)commutation can transform $w$ into $s_{p,q} u s_{p,q}$. The relation $ws_{p,q} = s_{p,q}w$ implies $us_{p,q} = s_{p,q}u$ in $J_n$, hence $w=s_{p,q} u s_{p,q}=u s_{p,q}s_{p,q} = u$ in $J_n$. We get a shorter word $u$ representing the same cactus as $w$, which contradicts the minimality of~$w$.
			
%			 Moreover, the words $us_{p,q}$ and $s_{p,q}u$ are irreducible since $w=s_{p,q} u s_{p,q}$ is so. Thus the above argument applies to the word $u$, and only Option 1 is realisable. The letter $s_{p,q}$ can then be (conjugation-)commuted all through the word $u$, and hence all through the word $w=s_{p,q} u s_{p,q}$. Again, on the $D_n$ side this means that all letters in the word $\od(w) \in FD_n$ commute with ${}^{s(c)}\!s_{p,q}$. 	
			%The above argument can now be applied to the word $u$ representing $c' \in J_n$. This time, $u$ cannot be transformed into $s_{p,q} v s_{p,q}$, since the relations $w= s_{p,q} u s_{p,q} = s_{p,q} s_{p,q} v s_{p,q} s_{p,q} = v$ yield a  representative $v$ of~$c$ shorter than $w$. The only remaining possibility is that the letter $s_{p,q}$ can be (conjugation-)commuted all through the word $u$, and hence all through the word $w=s_{p,q} u s_{p,q}$. Once again, on the $D_n$ side this means that all letters in the word $d(w) \in FD_n$ commute with the generator ${}^{s(c)}\!s_{p,q}$ in $D_n$.
		\end{description}
		
		Since these arguments work for any letter $s_{p,q}$, one concludes that the diagram $d(c) \in D_n$ represented by the word $\od(w) \in FD_n$ is central, as claimed.
	\end{proof}
	
	\begin{thm}\label{T:Center_Pure}
		The pure cactus subgroup $PJ_n$ has trivial centralizer in $J_n$ whenever $n>3$. In particular, its center is trivial. 
	\end{thm}
	
	The exceptional case $n=3$ can be treated by hand. We have
	\[J_3 \simeq FC_2 \rtimes \Z_2,\]
	where the free Coxeter subgroup $\langle s_{1,2}, s_{2,3} \rangle \simeq FC_2$ is generated by the $2$-leaf cacti, and the generator $s_{1,3}$ of the $\Z_2$ part acts on the $FC_2$ by permuting $s_{1,2}$ and $s_{2,3}$. Further,
	\[PJ_3 = \langle a:=s_{1,2}s_{2,3}s_{1,2}s_{1,3} \rangle \simeq \Z,\]
	and its centralizer in $J_3$ is 
	\[\mathrm{C}_{J_3} (PJ_3) = \langle b:=s_{1,2}s_{1,3} \rangle \simeq \Z.\]
	Note that $a=b^3$. See Appendix~\ref{A:PJ4} for more detail.
	
	\begin{proof}
		Let $w \in FJ_n$ be a word representing a non-trivial element $c \in J_n$ commuting with every pure cactus. It can be assumed of minimal length among such words.
		
		Any generator $s_{p,q}$ with leaf number $>2$ can be transformed into a pure cactus by attaching some $2$-leaf generators:
		\[\tilde{s}_{p,q} := s_{p,q} s_{p_1,p_1+1}s_{p_2,p_2+1}\cdots,\]
		since neighbouring transpositions generate the symmetric group $S_n$. This can be done in multiple ways; any choice will work for us. The commuting relation $c \tilde{s}_{p,q} = \tilde{s}_{p,q} c$ can be analysed along the lines of the proof of Theorem~\ref{T:Center}. One concludes that all letters in the word $\od(w) \in FD_n$ commute with ${}^{s(c)}\!s_{p,q}$ in $D_n$. Thus all letters in $\od(w)$ are \emph{almost friendly}, that is, commute with all the $\tau_I$ of size $|I| > 2$. 
		
		Let us now prove that $\tau_{1,2,\ldots,n}$ is the only almost friendly generator of $D_n$ when $n>3$. Indeed, given a proper subset $I \varsubsetneq \{ 1,2,\ldots,n \}$ of size $>2$, one can replace one of its elements with another element from $\{ 1,2,\ldots,n \}$, and get a subset $I'$ of size $>2$ such that $\tau_I$ and $\tau_{I'}$ do not commute in $D_n$. For a subset of size~$2$ the argument is similar, except that one replaces an element with two new ones; there is enough place for it in $\{ 1,2,\ldots,n \}$ since $n>3$.
		
		Thus the centralizer of $PJ_n$ can contain only the $d$-preimage $s_{1,n}$ of $\tau_{1,2,\ldots,n}$. But this element does not commute with $s_{1,3}s_{1,2}s_{2,3}s_{1,2} \in PJ_n$, since
		\begin{align*}
		d(s_{1,n}s_{1,3}s_{1,2}s_{2,3}s_{1,2}) &= \tau_{1,2,\ldots,n} \tau_{n-2,n-1,n} \tau_{n-2,n-1}\tau_{n-2,n}\tau_{n-1,n} \; \text{ and }\\
		d(s_{1,3}s_{1,2}s_{2,3}s_{1,2}s_{1,n}) &= \tau_{1,2,3}\tau_{2,3}\tau_{1,3}\tau_{1,2} \tau_{1,2,\ldots,n} =\tau_{1,2,\ldots,n}\tau_{1,2,3}\tau_{2,3}\tau_{1,3}\tau_{1,2}
		\end{align*}
		are distinct in~$D_n$ when $n>3$.
	\end{proof}
	
	\newpage %hhh
	\begin{thm}\label{T:Torsion}
		The cactus group $J_n$ has no odd torsion. 
	\end{thm}
	
	\begin{proof}
		Fix an odd prime $p$. Among non-trivial $p$-torsion elements in $J_n$ (if they exist), choose an element $c$ with the shortest possible representative~$w$. According to Proposition~\ref{P:cactus_normal}, the triviality of $c^p$ implies that in the word $w^p$ a letter $l$ can be (conjugation-)commuted to the right towards a letter $l'$, so that the two get annihilated. 
		
		\begin{description}
			\item[Case 1] The letters $l$ and $l'$ occupy different positions in the $q$th and $q'$th copies of $w$ respectively. We have $q<q'$ since $w$ is irreducible. One can assume $l$ to be in the last position, and $l'$ in the first position (otherwise the letters of $w$ should be (conjugation-)commuted accordingly). Remove the first letter of~$w$ and put it to the end; let $w'$ be the resulting word. It represents a non-trivial $p$-torsion element $c' \in J_n$ (which is a conjugate of~$c$). The word $w'^p$ is obtained from $w^p$ by moving the first letter to the end. In $w'^p$, the letters $l$ and $l'$ can still be (conjugation-)commuted together and annihilated. The letter $l$ remains in the $q$th copy of~$w'$, whereas the letter $l'$ is now in the $(q'-1)$st copy. They are still in different positions in their respective copies. Repeating this argument, one gets a non-trivial $p$-torsion element represented by a word $\tilde{w}$ with an annihilation possibility inside $\tilde{w}$ (case $q=q'$), hence with a representative shorter than~$w$. This contradicts the minimality of~$w$.	
			\item[Case 2] The letters $l$ and $l'$ occupy the same position $i$ in different copies of $w$. Consider the word 
			\[\od(w^p)=\od(w) \ {}^{t}\!\od(w) \ {}^{t^2}\!\od(w) \ldots \ {}^{t^{p-1}}\!\od(w),\]
			where $t=s(c)$. Then the letters of $\od(w^p) \in FD_n$ corresponding to the $p$ copies of the $i$th letter~$l$ from $w$ are $\tau, \ {}^{t}\!\tau, \ {}^{t^2}\!\tau, \ldots, \ {}^{t^{p-1}}\!\tau$. Since $p$ is prime, the permutation $t$ is of order~$p$ or~$1$ (as $t^p=s(c)^p=s(c^p)=s(1)=\Id$). In its orbit containing $\tau$, two elements, corresponding to $l$ and $l'$ in $w^p$, can be commuted together and annihilated, thus coincide. The $p$ letters in $\od(w^p)$ corresponding to~$l$ are thus all identical, and can be commuted all through the word $\od(w^p)$. Since the word $\od(w^p)$ represents the trivial Gauss diagram, it contains an even number of copies of the letter $\tau$, and thus at least one copy different from the $p$ copies mentioned above; here we used that $p$ is odd for the first time in this proof. Thus $\tau$ appears at least twice in one of the words $\od(w), \ {}^{t}\!\od(w), \ {}^{t^2}\!\od(w), \ldots, \ {}^{t^{p-1}}\!\od(w)$, where its two occurrences can be moved together and annihilated. This contradicts the minimality of~$w$. \qedhere
		\end{description} 
	\end{proof}
	
	\begin{thm}\label{T:Torsion_Even}
		The cactus group $J_{2^k}$ has torsion of order~$2^k$. 
	\end{thm}
	
	\begin{proof}
		Consider the cacti defined inductively by
		\begin{align*}
		t_1&=s_{1,2},\\
		t_2&=s_{1,2}s_{1,4},	\hspace*{1cm}	\cdots\\
		t_{k+1}&=t_k s_{1,2^{k+1}}.
		\end{align*}
		The cactus $t_k$ is defined in the group $J_n$ whenever $n \geq 2^k$. Let us prove by induction that $t_k$ is of order~$2^k$. For $k=1$, this is just the idempotence of $s_{1,2}$. To move from $k$ to $k+1$, observe that
		\begin{align*}
		t_{k+1}^2&=(t_k s_{1,2^{k+1}})^2= t_k (s_{1,2^{k+1}}t_k s_{1,2^{k+1}}) = t_k t'_k.
		\end{align*}
		In the word $t'_k:=s_{1,2^{k+1}}t_k s_{1,2^{k+1}}$, the first letter $s_{1,2^{k+1}}$ can be conjugate-commuted all the way to the right and annihilated with the last letter. In the resulting word, the indices of all letters are $>2^k$. Thus the cactus $t_k$ and its conjugate $t'_k$ commute. Since both are of order $2^k$ by assumption, so is their product $t_{k+1}^2$. Hence the order of $t_{k+1}$ is $2^{k+1}$.
	\end{proof}
	
	\begin{thm}\label{T:Torsion_Pure}
		The pure cactus group $PJ_n$ is torsionless. 
	\end{thm}
	
	\begin{proof}
		By Theorem~\ref{T:Torsion}, it is sufficient to show that $PJ_n$ has no $2$-torsion. Assume that there is some. Among non-trivial $2$-torsion elements in $PJ_n$, choose an element $c$ with the shortest possible representative~$w$. Following the proof of Theorem~\ref{T:Torsion}, one concludes that $\od(w)$ is a product of pairwise commuting generators. In particular, one can reorganise the word~$w$ by (conjugation-)commutation into a word~$w'$ so that the leaf number of its letters never increases from left to right. Let $s_{p,q}$ be the first letter of~$w$. Viewing $s(c)$ as a permutation on the set $\{1,2,\ldots,n\}$, let us trace what it does to the element $p$. First, $s(s_{p,q})$ sends $p$ to position $q$. The next letter whose associated permutation moves this element has to be of the form $s_{p',q}$ with $p' > p$, due to the pairwise commutativity of the letters of $\od(w')$. The next letter moving this element is $s_{p',q'}$ with $q'< q$, and so on. We observe a retracting ping-pong-like trajectory. Overall, the permutation $s(c)$ moves our element strictly to the right, and thus cannot be trivial. Hence the cactus~$c$ cannot be pure.
	\end{proof}
	
%	\newpage %hhh
	\appendix
	\section{A one-relator presentation for $PJ_4$}\label{A:PJ4}
	
	The goal of this appendix is to provide explicit group presentations for $PJ_3$ and $PJ_4$.
	First let us state simpler presentations for $J_3$ and $J_4$ which can be easily obtained using Tietze transformations:
	
	\begin{prop}\label{theoreme5}
		The cactus groups $J_3$ and $J_4$ admit the following group presentations:
		\begin{align*}
		J_3 &\simeq \langle s_{1,2}, s_{1,3} \,  |  \, s_{1,2}^2=s_{1,3}^2=1\rangle \cong \Z_2 \ast \Z_2\\
		J_4 &\simeq  \Biggl\langle 
		\begin{array}{l|c}
		& s_{1,2}^2=s_{1,3}^2=s_{1,4}^2=1, \\
		s_{1,2}, s_{1,3}, s_{1,4}  &  s_{1,2} s_{1,4} s_{1,2} s_{1,4}= s_{1,4} s_{1,2} s_{1,4} s_{1,2},  \\
		& s_{1,4} s_{1,3}s_{1,2}s_{1,3}= s_{1,3}s_{1,2}s_{1,3} s_{1,4}                                             
		\end{array}
		\Biggr\rangle.
		\end{align*}	
	\end{prop}

Similar presentations can be produced for general~$n$.
	
	\begin{cor}\label{C:PJ3}
		The pure cactus group $PJ_3$ admits the following group presentation:
		\[PJ_3= \langle  (s_{1,2} s_{1,3})^3 \rangle \simeq \Z.\]
	\end{cor}
	
	\begin{proof}
		The cactus group $J_3$ is the RACG  $\Z_2 \ast \Z_2$ generated by $s_{1,2}$ and $s_{1,3}$, and the symmetric group $S_3$ admits the presentation $S_3 \simeq \langle  s_1, s_2  \mid s_1^2=s_2^2=1, (s_1 s_2)^3=1 \rangle$, which, when $s_2$ is replaced with $s'_2=s_1s_2 s_1$, becomes $S_3 \simeq \langle  s_1, s'_2  \mid s_1^2=(s'_2)^2=1, (s_1 s'_2)^3=1 \rangle$. Since $s(s_{1,2})=s_1$ and $s(s_{1,3})=s_1s_2 s_1 = s'_2$, the kernel of $s$ is freely generated by $(s_{1,2} s_{1,3})^3$.
	\end{proof}

Note that the generator of $PJ_3$ above can be rewritten in a shorter form:
\[a:=(s_{1,2} s_{1,3})^3=s_{1,2} (s_{1,3}s_{1,2} s_{1,3})s_{1,2} s_{1,3}=s_{1,2} s_{2,3}s_{1,2} s_{1,3}.\]
	
\begin{thm}\label{theoreme11}
	The pure cactus group $PJ_4$ admits the following group presentation:
	\[PJ_4=\langle \alpha, \ \beta, \ \gamma, \ \delta, \ \epsilon \ | \ 
\alpha \gamma \epsilon \beta \epsilon \alpha^{-1} \delta^{-1} \beta \gamma  \delta^{-1}=1
\rangle,\]
\begin{align*}
\text{where } 
\alpha &= (s_{1,3} s_{1,2})^3,\\
\beta&= s_{1,2} s_{1,3} s_{1,4} s_{1,3} s_{1,4} s_{1,2} s_{1,4} = s_{1,3} s_{1,4} s_{1,3} (s_{1,2}s_{1,4})^2,\\
\gamma&= s_{1,2} s_{1,4} s_{1,2} (s_{1,3}s_{1,4})^2,\\
\delta&= s_{1,3} (s_{1,2} s_{1,4})^2 s_{1,3} s_{1,4} ,\\
\epsilon&= (s_{1,4} s_{1,2} s_{1,3} s_{1,2})^2.
\end{align*}
\end{thm}

We will derive this presentation by hand, using the Reidemeister--Schreier method. Our computations were verified in GAP by Neha Nanda and John Guaschi.

The group $PJ_4$ is thus a one-relator group, where the relation is not a power.
Applying Theorem 4.12 of \cite{CombGroupTheory}, we then obtain another proof of the absence of torsion in $PJ_4$. One can also derive several other nice properties of $PJ_4$, using the classical theory of one-relator groups (see for instance \cite{CombGroupTheory,CombGroupTheoryLS,OneRelatorPutman} and references therein): $PJ_4$ is locally indicable and of cohomological dimension $\leq 2$; it has algorithmically decidable word problem; it satisfies the Tits alternative: every its subgroup is either solvable or contains a free group of rank $2$; its presentation $2$-complex is aspherical.

Note that in our presentation of $PJ_4$, we used the generating set $\{s_{1,2}, s_{1,3}, s_{1,4}\}$ of $J_4$. One obtains shorter and more manageable expressions by including the generators $s_{i,j}$ with $i>1$:
\begin{align*}
\alpha &= (s_{1,3} s_{1,2})^3= s_{1,3} s_{1,2} s_{2,3}s_{1,2},&
\beta& =  s_{1,4} s_{2,4} s_{1,3} s_{1,2} s_{3,4},\\
\gamma&= s_{1,2} s_{3,4} s_{2,4} s_{1,3}  s_{1,4},&
\delta&= s_{1,3} s_{1,2} s_{3,4}  s_{1,3} s_{1,4} ,\\
\epsilon&=  s_{3,4} s_{2,3} s_{2,4} s_{1,2} s_{2,3} s_{1,3}.&&
\end{align*}
In particular, $\alpha$ is the inverse of the (image of the) generator $a$ of $PJ_3$ from Corollary~\ref{C:PJ3}. Also, some generators can be replaced with shorter and/or more meaningful ones:
\begin{enumerate}
	\item $\epsilon \leadsto \zeta= \epsilon \alpha^{-1}=s_{2,4} s_{2,3} s_{3,4} s_{2,3}$, which is the generator $\alpha$ ``shifted'' to the right (in other words, the inclusion of $J_3$ into $J_4$ given by $s_{p,q} \mapsto s_{p+1,q+1}$ sends $a^{-1}$ to $\zeta$);
	\item $\gamma \leadsto \eta= \beta \gamma = (s_{1,3}  s_{2,4})^2$, which is the commutator of $s_{1,3}$ and $s_{2,4}$;
	\item $\delta \leadsto \theta = \alpha^{-1} \delta =s_{1,2}s_{2,3}s_{1,2}s_{1,3}s_{1,3} s_{1,2} s_{3,4} s_{1,3} s_{1,4} = s_{1,2} s_{2,3} s_{3,4} s_{1,3} s_{1,4}$;
	\item 
	\begin{align*}
\beta \leadsto \kappa = \theta \eta^{-1} \beta &=s_{1,2} s_{2,3} s_{3,4} s_{1,3} s_{1,4} \cdot s_{2,4}s_{1,3} s_{2,4}s_{1,3} \cdot s_{1,4} s_{2,4} s_{1,3} s_{1,2} s_{3,4} \\
&=	s_{1,2} s_{2,3} s_{3,4}s_{1,4} s_{2,4} \cdot s_{2,4}s_{1,3} s_{2,4}s_{1,3} \cdot  s_{1,3} s_{2,4} s_{1,4} s_{1,2} s_{3,4} \\
&=	s_{1,2} s_{2,3} s_{3,4}s_{1,4} s_{1,3} s_{1,4} s_{1,2} s_{3,4} \\
&=	s_{1,2} s_{2,3} s_{3,4} s_{2,4} s_{3,4} s_{1,2} \\
&=	s_{1,2} \cdot s_{2,3} s_{3,4} s_{2,3} s_{2,4} \cdot s_{1,2},
	\end{align*}
 which is $\zeta^{-1}$ conjugated by $s_{1,2}$.	
\end{enumerate}
The generators $\alpha, \zeta, \kappa, \theta$ and $\eta$ are depicted in Fig.~\ref{Pic:PJ4_gen}.
\begin{figure}[h]
	\begin{center}
	\begin{tikzpicture}[line cap=round,line join=round,x=1cm,y=.6cm,rounded corners=2pt,line width=1.5pt]
\draw (0.,0.)-- (0.3,0.)-- (1.7,2.)-- (2.,2.)-- (4.,0.)-- (5.,0.);
\draw (0.,2.)-- (0.3,2.)-- (1.7,0.)--(2.,0.)-- (3.,0.)-- (5.,2.);
\draw (0.,1.)-- (2.,1.)-- (3.,2.)-- (4.,2.)-- (5.,1.);
\draw (0.,-1.)-- (5.,-1.);
\draw (-1,0.5) node {$\alpha=$};
	\end{tikzpicture}\hspace*{1cm} 
	\begin{tikzpicture}[line cap=round,line join=round,x=1cm,y=.6cm,rounded corners=2pt,line width=1.5pt]
\draw (0.,0.)-- (0.3,0.)-- (1.7,2.)-- (2.,2.)-- (4.,0.)-- (5.,0.);
\draw (0.,2.)-- (0.3,2.)-- (1.7,0.)--(2.,0.)-- (3.,0.)-- (5.,2.);
\draw (0.,1.)-- (2.,1.)-- (3.,2.)-- (4.,2.)-- (5.,1.);
\draw (0.,3.)-- (5.,3.);
\draw (-1,1.5) node {$\zeta=$};
\end{tikzpicture}

\bigskip\bigskip
	\begin{tikzpicture}[line cap=round,line join=round,x=1cm,y=.6cm,rounded corners=2pt,line width=1.5pt]
\draw (1.,3.)--(2.,2.)-- (4.,0.)-- (5.3,0.)-- (6.7,2.)-- (7.4,3.);
\draw (1.,0.)-- (3.,0.)-- (5.,2.)-- (5.3,2.)-- (6.7,0.)-- (7.4,0.);
\draw (1.,1.)--(2.,1.)-- (3.,2.)-- (4.,2.)-- (5.,1.)-- (7.4,1.);
\draw (1.,2.)--(2.,3.)-- (6.4,3.)-- (7.4,2.);
\draw (0,1.5) node {$\kappa=$};
\end{tikzpicture}\hspace*{.7cm} 
	\begin{tikzpicture}[line cap=round,line join=round,x=1cm,y=.6cm,rounded corners=2pt,line width=1.5pt]
\draw (0.3,2.)-- (3.3,-1.)-- (5.3,-1.)-- (6.7,2.)-- (7,2.);
\draw (0.3,1.)-- (1.3,2.)-- (3.3,2.)--(4.7,0.)-- (5.3,0.)-- (6.7,1.)-- (7,1.);
\draw (0.3,0.)-- (1.3,0.)-- (2.3,1.)-- (5.3,1.)-- (6.7,0.)-- (7,0.);
\draw (0.3,-1.)-- (2.3,-1.)-- (3.3,0.)--(4.7,2.)-- (5.3,2.)-- (6.7,-1.)-- (7,-1.);
\draw (-0.7,0.5) node {$\theta=$};
\end{tikzpicture}

\bigskip\bigskip
\begin{tikzpicture}[line cap=round,line join=round,x=1cm,y=.6cm,rounded corners=2pt,line width=1.5pt]
\draw (0.,0.)-- (0.3,0.)-- (1.7,2.)-- (4.3,2.)-- (5.7,0.)-- (8,0.);
\draw (0.,2.)-- (0.3,2.)-- (1.7,0.)--(4.3,0.)-- (5.7,2.)-- (8.,2.);
\draw (0.,1.)-- (2.3,1.)-- (3.7,-1.)-- (6.3,-1.)-- (7.7,1.)-- (8,1.);
\draw (0.,-1.)--(2.3,-1.)-- (3.7,1.)-- (6.3,1.)-- (7.7,-1.)-- (8.,-1.);
\draw (-1,0.5) node {$\eta=$};
\end{tikzpicture}			
		\caption{Generators of the group $PJ_4$}\label{Pic:PJ4_gen}
	\end{center}
\end{figure}
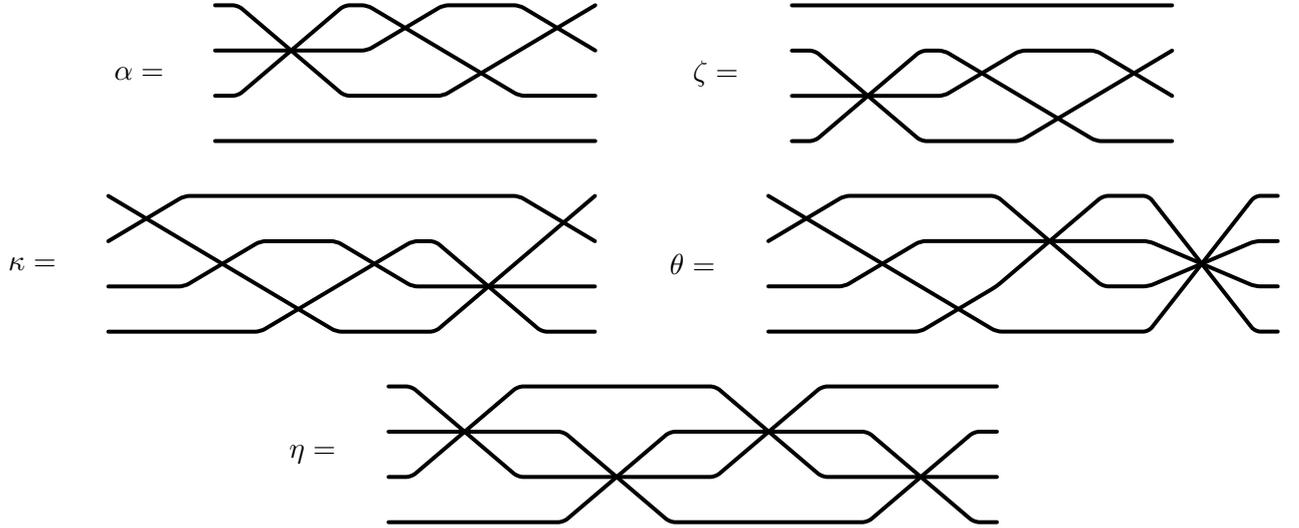

Observe that the squares $\alpha^2, \zeta^2$ and $\kappa^2$ can be rewritten using $2$-leaf generators only, and yield $3$ out of the $7$ free generators of the \emph{pure twin group} (also called the \emph{planar pure braid group}) $PTw_4$ from \cite{Mostovoy_PureTwin}.

\begin{rem}
	In \cite{Devadoss_MosaicOp}, the group $PJ_4$ was given a topological interpretation. It is the fundamental group of the connected sum of five real projective planes. This yields its one-relator presentation of the following form: 
		\[PJ_4=\langle \ \alpha_1, \ \alpha_2, \ \alpha_3, \ \alpha_4, \ \alpha_5 \ | \ 
	\alpha_1^2 \alpha_2^2 \alpha_3^2 \alpha_4^2 \alpha_5^2=1
	\ \rangle.\]
	However, it seems difficult to find explicit expressions of the generators $\alpha_k$ in terms of the generators $s_{i,j}$ of the whole cactus group $J_4$.
\end{rem}

	\begin{proof}[Proof of Theorem~\ref{theoreme11}]
	Let us first recall the Reidemeister--Schreier method, in order to fix notations.
			
	Let  $G$ be a group with presentation $G=\langle \, X \mid R \, \rangle$, where  $X=\{x_1, \ldots , x_p\}$ is the set of generators and  $R=\{r_1,\ldots, r_q\}$ is the set of relations. Let $F(X)$ be the free group on $X$. A \emph{Schreier transversal} of a subgroup $H$ in $G$ is a set $T$ of reduced words in the generators $\{x_1,\ldots,x_p\}$ containing exactly one representative of every right coset of $H$, and together with each word containing all its prefixes. Any subgroup of $G$ admits a Schreier transversal \cite{CombGroupTheory}.
	
	Fix a Schreier transversal $T$ of $H$. Denote by  $^-$ the map $F(X) \to T$ sending $w$ to its representative $\overline{w} \in T$. For any $k\in T$ and $x_i \in X$, put \[a_{k,x_i}=(k x_i)(\overline{k x_i})^{-1}.\]  According to Theorem 2.9 of \cite{CombGroupTheory}, $H$ admits a group presentation having as generators all the non-trivial $a_{k,x_i}$. A system of relations is constructed as follows. Let $w = x_{i_1}^{\varepsilon_1} x_{i_1}^{\varepsilon_2} \ldots
	x_{i_m}^{\varepsilon_{m}}$, where
	$\varepsilon_l = \pm 1\text{ and } 
	x_{i_l} \in X$ for $l=1, \ldots, m$. The \emph{rewriting function} is defined as
	\[\tau(w) = a_{k_{i_1},x_{i_1}}^{\varepsilon_1} \,  a_{k_{i_2},x_{i_2}}^{\varepsilon_2}
	\ldots a_{k_{i_m},x_{i_m}}^{\varepsilon_{m}},\]
	\[ \text{where } k_{i_j}= \begin{cases}
	\overline{x_{i_1}^{\varepsilon_1} \cdots x_{i_{j-1}}^{\varepsilon_{j-1}}} & \text{ if } \varepsilon_j = 1,\\
	\overline{x_{i_1}^{\varepsilon_1} \cdots x_{i_{j}}^{\varepsilon_{j}}} & \text{ if } \varepsilon_j = -1.
	\end{cases} \]
	A complete set of relations for $H$ is given by $ \{\tau(k r_j k^{-1})\ | \ 1\leq j \leq q, \ k \in T\}$.
	
We now turn to our concrete subgroup $H=PJ_4$ of $G=J_4$. Fix the following Schreier transversal:
		\[K=\left\{\begin{array}{lllllll}
		k_1= 1, \ k_2= s_{1,2}, \ k_3 = s_{1,3}, \ k_4 = s_{1,4}, \ k_5 = s_{1,2}s_{1,3}, \ k_6 = s_{1,2}s_{1,4}, \ k_7=s_{1,3}s_{1,2}, \\
		k_8=s_{1,3}s_{1,4}, \ k_9= s_{1,4}s_{1,2}, \ k_{10}=s_{1,4}s_{1,3}, \ k_{11}=s_{1,2}s_{1,3}s_{1,2}, \ k_{12}=s_{1,2}s_{1,3}s_{1,4}, \\ 
		k_{13}=s_{1,2}s_{1,4}s_{1,2},\ k_{14}=s_{1,2}s_{1,4}s_{1,3}, \ k_{15}=s_{1,3}s_{1,2}s_{1,4}, \ k_{16}=s_{1,3}s_{1,4}s_{1,2}, \\
		k_{17}=s_{1,3}s_{1,4}s_{1,3}, \ k_{18}=s_{1,4}s_{1,2}s_{1,3}, \ k_{19}=s_{1,4}s_{1,2}s_{1,4}, \ k_{20}= s_{1,4}s_{1,3}s_{1,2}, \\ 
		k_{21}= s_{1,4}s_{1,3}s_{1,4}, \ k_{22}=s_{1,2}s_{1,3}s_{1,2}s_{1,4}, \ k_{23}= s_{1,2}s_{1,3}s_{1,4}s_{1,2}, \ k_{24}= s_{1,2}s_{1,4}s_{1,3}s_{1,2}
		\end{array}\right\}.\]
		
		The non-trivial generators of $PJ_4$ are:
		
		\noindent $\begin{array}{lllll}
		a_{k_7,s_{1,3}} &=&  s_{1,3}s_{1,2}s_{1,3}s_{1,2}s_{1,3}s_{1,2}\\
		%&=&a_{k_1,s_{1,3}}a_{k_3,s_{1,2}}a_{k_7,s_{1,3}}a_{k_{11},s_{1,2}}a_{k_5,s_{1,3}}a_{k_2,s_{1,2}}
		\end{array} $\\  
		$\begin{array}{llll}
		a_{k_{11},s_{1,3}}&=& s_{1,2}s_{1,3}s_{1,2}s_{1,3}s_{1,2}s_{1,3}\\
		%&=&a_{k_1,s_{1,2}}a_{k_2,s_{1,3}}a_{k_{5},s_{1,2}}a_{k_{11},s_{1,3}}a_{k_7,s_{1,2}}a_{k_3,s_{1,3}}
		\end{array} $\\ 
		$\begin{array}{llllll}
		a_{k_{12},s_{1,3}}&=&s_{1,2}s_{1,3}s_{1,4}s_{1,3}s_{1,4}s_{1,2}s_{1,4}\\
		%&=&a_{k_1,s_{1,2}}a_{k_2,s_{1,3}}a_{k_5,s_{1,4}}a_{k_{12},s_{1,3}}a_{k_{19},s_{1,4}}a_{k_9,s_{1,2}}a_{k_4,s_{1,4}}
		\end{array} $\\
		$\begin{array}{llllll}
		a_{k_{13},s_{1,3}}&=&s_{1,2}s_{1,4}s_{1,2}s_{1,3}s_{1,4}s_{1,3}s_{1,4}\\
		%&=&a_{k_1,s_{1,2}}a_{k_2,s_{1,4}}a_{k_6,s_{1,2}}a_{k_{13},s_{1,3}}a_{k_{21},s_{1,4}}a_{k_{10},s_{1,3}}a_{k_4,s_{1,4}}
		\end{array} $\\
		$\begin{array}{llllll}
		a_{k_{13},s_{1,4}}&=&s_{1,2}s_{1,4}s_{1,2}s_{1,4}s_{1,3}s_{1,4}s_{1,3}\\
		%&=&a_{k_1,s_{1,2}}a_{k_2,s_{1,4}}a_{k_6,s_{1,2}}a_{k_{13},s_{1,4}}a_{k_{17},s_{1,3}}a_{k_8,s_{1,4}}a_{k_3,s_{1,3}}
		\end{array} $\\ 
		$\begin{array}{lllllll}
		a_{k_{14},s_{1,4}}&=&s_{1,2}s_{1,4}s_{1,3}s_{1,4}s_{1,3}s_{1,2}s_{1,4}\\
		%&=&a_{k_1,s_{1,2}}a_{k_2,s_{1,4}}a_{k_6,s_{1,3}}a_{k_{14},s_{1,4}}a_{k_{18},s_{1,3}}a_{k_9,s_{1,2}}a_{k_4,s_{1,4}}
		\end{array} $\\
		$\begin{array}{lllll}
		a_{k_{15},s_{1,2}}&=&s_{1,3}s_{1,2}s_{1,4}s_{1,2}s_{1,4}s_{1,3}s_{1,4} \\
		%&=& a_{k_1,s_{1,3}}a_{k_3,s_{1,2}}a_{k_7,s_{1,4}}a_{k_{15},s_{1,2}}a_{k_{21},s_{1,4}}a_{k_{10},s_{1,3}}a_{k_4,s_{1,4}}
		\end{array} $\\
		$\begin{array}{lllllll}
		a_{k_{15},s_{1,3}}&=&s_{1,3}s_{1,2}s_{1,4}s_{1,3}s_{1,2}s_{1,3}s_{1,4}s_{1,2} \\
		%&=& a_{k_1,s_{1,3}}a_{k_3,s_{1,2}}a_{k_7,s_{1,4}}a_{k_{15},s_{1,3}}a_{k_{24},s_{1,2}}a_{k_{14},s_{1,3}}a_{k_6,s_{1,4}}a_{k_2,s_{1,2}}
		\end{array} $\\ 
		$\begin{array}{llllllll}
		a_{k_{16},s_{1,3}}&=&s_{1,3}s_{1,4}s_{1,2}s_{1,3}s_{1,2}s_{1,4}s_{1,3}s_{1,2}\\
		%&=&a_{k_1,s_{1,3}}a_{k_3,s_{1,4}}a_{k_8,s_{1,2}}a_{k_{16},s_{1,3}}a_{k_{23},s_{1,2}}a_{k_{12},s_{1,4}}a_{k_5,s_{1,3}}a_{k_2,s_{1,2}}
		\end{array} $\\
		$\begin{array}{lllllll}
		a_{k_{16},s_{1,4}}&=&s_{1,3}s_{1,4}s_{1,2}s_{1,4}s_{1,2}s_{1,3}s_{1,4}\\
		%&=&a_{k_1,s_{1,3}}a_{k_3,s_{1,4}}a_{k_8,s_{1,2}}a_{k_{16},s_{1,4}}a_{k_{20},s_{1,2}}a_{k_{10},s_{1,3}}a_{k_4,s_{1,4}}
		\end{array} $\\
		$\begin{array}{llllllll}
		a_{k_{17},s_{1,2}}&=&s_{1,3}s_{1,4}s_{1,3}s_{1,2}s_{1,4}s_{1,2}s_{1,4}\\
		%&=&a_{k_1,s_{1,3}}a_{k_3,s_{1,4}}a_{k_8,s_{1,3}}a_{k_{17},s_{1,2}}a_{k_{19},s_{1,4}}a_{k_9,s_{1,2}}a_{k_4,s_{1,4}}
		\end{array} $\\ 
		$\begin{array}{lllllll}
		a_{k_{17},s_{1,4}}&=&s_{1,3}s_{1,4}s_{1,3}s_{1,4}s_{1,2}s_{1,4}s_{1,2}\\
		%&=&a_{k_1,s_{1,3}}a_{k_3,s_{1,4}}a_{k_8,s_{1,3}}a_{k_{17},s_{1,4}}a_{k_{13},s_{1,2}}a_{k_6,s_{1,4}}a_{k_2,s_{1,2}}
		\end{array} $\\
		$\begin{array}{lllll}
		a_{k_{18},s_{1,2}}&=&s_{1,4}s_{1,2}s_{1,3}s_{1,2}s_{1,4}s_{1,2}s_{1,3}s_{1,2}\\
		%&=&a_{k_1,s_{1,4}}a_{k_4,s_{1,2}}a_{k_9,s_{1,3}}a_{k_{18},s_{1,2}}a_{k_{22},s_{1,4}}a_{k_{11},s_{1,2}}a_{k_5,s_{1,3}}a_{k_2,s_{1,2}}
		\end{array} $\\
		$\begin{array}{lllllll}
		a_{k_{18},s_{1,4}}&=&s_{1,4}s_{1,2}s_{1,3}s_{1,4}s_{1,3}s_{1,4}s_{1,2}\\
		%&=&a_{k_1,s_{1,4}}a_{k_4,s_{1,2}}a_{k_9,s_{1,3}}a_{k_{18},s_{1,4}}a_{k_{14},s_{1,3}}a_{k_6,s_{1,4}}a_{k_2,s_{1,2}}
		\end{array} $\\
		$\begin{array}{lllllll}
		a_{k_{19},s_{1,2}}&=&s_{1,4}s_{1,2}s_{1,4}s_{1,2}s_{1,3}s_{1,4}s_{1,3}\\
		%&=&a_{k_1,s_{1,4}}a_{k_4,s_{1,2}}a_{k_9,s_{1,4}}a_{k_{19},s_{1,2}}a_{k_{17},s_{1,3}}a_{k_8,s_{1,4}}a_{k_3,s_{1,3}}
		\end{array} $\\
		$\begin{array}{lllllll}
		a_{k_{19},s_{1,3}}&=&s_{1,4}s_{1,2}s_{1,4}s_{1,3}s_{1,4}s_{1,3}s_{1,2}\\
		%&=&a_{k_1,s_{1,4}}a_{k_4,s_{1,2}}a_{k_9,s_{1,4}}a_{k_{19},s_{1,3}}a_{k_{12},s_{1,4}}a_{k_5,s_{1,3}}a_{k_2,s_{1,2}}
		\end{array} $\\ 
		$\begin{array}{llllllll}
		a_{k_{20},s_{1,3}}&=&s_{1,4}s_{1,3}s_{1,2}s_{1,3}s_{1,4}s_{1,2}s_{1,3}s_{1,2}\\
		%&=&a_{k_1,s_{1,4}}a_{k_4,s_{1,3}}a_{k_{10},s_{1,2}}a_{k_{20},s_{1,3}}a_{k_{22},s_{1,4}}a_{k_{11},s_{1,2}}a_{k_5,s_{1,3}}a_{k_2,s_{1,2}}
		\end{array} $\\
		$\begin{array}{lllll}
		a_{k_{20},s_{1,4}}&=&s_{1,4}s_{1,3}s_{1,2}s_{1,4}s_{1,2}s_{1,4}s_{1,3}\\
		%&=&a_{k_1,s_{1,4}}a_{k_4,s_{1,3}}a_{k_{10},s_{1,2}}a_{k_{20},s_{1,3}}a_{k_{16},s_{1,2}}s_ {k_8,s_{1,4}}a_{k_3,s_{1,3}}
		\end{array} $\\
		$\begin{array}{llllll}
		a_{k_{21},s_{1,2}}&=&s_{1,4}s_{1,3}s_{1,4}s_{1,2}s_{1,4}s_{1,2}s_{1,3}\\
		%&=&a_{k_1,s_{1,4}}a_{k_4,s_{1,3}}a_{k_{10},s_{1,4}}a_{k_{21},s_{1,2}}a_{k_{15},s_{1,4}}a_{k_7,s_{1,2}}a_{k_3,s_{1,3}}
		\end{array} $\\
		$\begin{array}{lllllllll}
		a_{k_{21},s_{1,3}}&=&s_{1,4}s_{1,3}s_{1,4}s_{1,3}s_{1,2}s_{1,4}s_{1,2}\\
		%&=&a_{k_1,s_{1,4}}a_{k_4,s_{1,3}}a_{k_{10},s_{1,4}}a_{k_{21},s_{1,3}}a_{k_{13},s_{1,2}}a_{k_6,s_{1,4}}a_{k_2,s_{1,2}}
		\end{array} $\\ 
		$\begin{array}{lllllll}
		a_{k_{22},s_{1,2}}&=&s_{1,2}s_{1,3}s_{1,2}s_{1,4}s_{1,2}s_{1,3}s_{1,2}
		s_{1,4}\\
		%&=&a_{k_1,s_{1,2}}a_{k_2,s_{1,3}}a_{k_5,s_{1,2}}a_{k_{11},s_{1,4}}a_{k_{22},s_{1,2}}a_{k_{18},s_{1,3}}a_{k_9,s_{1,2}}a_{k_4,s_{1,4}}
		\end{array} $\\
		$\begin{array}{lllllll}
		a_{k_{22},s_{1,3}}&=&s_{1,2}s_{1,3}s_{1,2}s_{1,4}s_{1,3}s_{1,2}s_{1,3}
		s_{1,4}\\
		%&=&a_{k_1,s_{1,2}}a_{k_2,s_{1,3}}a_{k_5,s_{1,2}}a_{k_{11},s_{1,4}}a_{k_{22},s_{1,3}}a_{k_{20},s_{1,2}}a_{k_{10},s_{1,3}}a_{k_4,s_{1,4}}
		\end{array} $\\
		$\begin{array}{lllllll}
		a_{k_{23},s_{1,3}}&=&s_{1,2}s_{1,3}s_{1,4}s_{1,2}s_{1,3}s_{1,2}
		s_{1,4}s_{1,3}\\
		%&=&a_{k_1,s_{1,2}}a_{k_2,s_ {1,3}}a_{k_5,s_{1,4}}a_{k_{12},s_{1,2}}a_{k_{23},s_{1,3}}a_{k_{16},s_{1,2}}a_{k_8,s_{1,4}}a_{k_3,s_{1,3}}
		\end{array} $\\
		$\begin{array}{llllllll}
		a_{k_{23},s_{1,4}}&=&s_{1,2}s_{1,3}s_{1,4}s_{1,2}s_{1,4}s_{1,2}s_{1,3}s_{1,4}s_{1,2}\\
		%&=&a_{k_1,s_{1,2}}a_{k_2,s_ {1,3}}a_{k_5,s_{1,4}}a_{k_{12},s_{1,2}}a_{k_{23},s_{1,4}}a_{k_{24},s_{1,2}}a_{k_{14},s_{1,3}}a_{k_6,s_{1,4}}a_{k_2,s_{1,2}}
		\end{array} $\\
		$\begin{array}{lllllll}
		a_{k_{24},s_{1,3}}&=&s_{1,2}s_{1,4}s_{1,3}s_{1,2}s_{1,3}s_{1,4}s_{1,2}
		s_{1,3}\\
		%&=&a_{k_1,s_{1,2}}a_{k_2,s_{1,4}}a_{k_6,s_{1,3}}a_{k_{14},s_{1,2}}a_{k_{24},s_{1,3}}a_{k_{15},s_{1,4}}a_{k_7,s_{1,2}}a_{k_3,s_{1,3}}
		\end{array} $\\
		$\begin{array}{llllllll}
		a_{k_{24},s_{1,4}}&=&s_{1,2}s_{1,4}s_{1,3}s_{1,2}s_{1,4}s_{1,2}s_{1,4}s_{1,3}s_{1,2}\\
		%&=&a_{k_1,s_{1,2}}a_{k_2,s_{1,4}}a_{k_6,s_{1,3}}a_{k_{14},s_{1,2}}a_{k_{24},s_{1,4}}a_{k_{23},s_{1,2}}a_{k_{12},s_{1,4}}a_{k_5,s_{1,3}}a_{k_2,s_{1,2}}
		\end{array} $\\
		
		We detail only non-trivial relations of type  \textit{$\tau(kRk^{-1})$}:
		
		%$\begin{array}{lllllllll}
		%\tau(k_1s_{1,2}^2k_1^{-1})&=&\tau(s_{1,2}s_{1,2})=a_{k_1,s_{1,2}}a_{k_2,s_{1,2}}=e
		%\end{array} $\\
		%$\begin{array}{lllllllll}
		%\tau(k_1s_{1,3}^2k_1^{-1})&=&\tau(s_{1,3}s_{1,3})=a_{k_1,s_{1,3}}a_{k_2,s_{1,3}}=e
		%\end{array} $\\
		%$\begin{array}{lllllll}
		%\tau(k_1s_{1,4}^2k_1^{-1})&=&\tau(s_{1,4}s_{1,4})= a_{k_1,s_{1,4}}a_{k_2,s_{1,4}}=e
		%\end{array} $\\

		$\begin{array}{llllllll}
		\tau(k_1(s_{1,2}s_{1,3}s_{1,4}s_{1,3})^2 k_1^{-1})=\tau(s_{1,2}s_{1,3}s_{1,4}s_{1,3}s_{1,2}s_{1,3}s_{1,4}s_{1,3})\\
		=a_{k_1,s_{1,2}}a_{k_2,s_{1,3}}a_{k_5,s_{1,4}}a_{k_{12},s_{1,3}}a_{k_{19},s_{1,2}}a_{k_{17},s_{1,3}}a_{k_8,s_{1,4}}a_{k_3,s_{1,3}}=
		a_{k_{12},s_{1,3}}a_{k_{19},s_{1,2}}
		\end{array} $\\
		
		$\begin{array}{llllllll}
		\tau(k_1(s_{1,2}s_{1,4})^4k_1^{-1})=\tau(s_{1,2}s_{1,4}s_{1,2}s_{1,4}s_{1,2}s_{1,4}s_{1,2}s_{1,4})\\
		=a_{k_1,s_{1,2}}a_{k_2,s_{1,4}}a_{k_6,s_{1,2}}a_{k_{13},s_{1,4}}a_{k_{17},s_{1,2}}a_{k_{19},s_{1,4}}a_{k_9,s_{1,2}}a_{k_4,s_{1,4}}=
		a_{k_{13},s_{1,4}}a_{k_{17},s_{1,2}}
		\end{array} $\\
		
		%$\begin{array}{llllll}
		%\tau(k_2s_{1,2}^2k_2^{-1})&=&\tau(s_{1,2}s_{1,2}s_{1,2}s_{1,2})=a_{k_1,s_{1,2}}a_{k_2,s_{1,2}}a_{k_1,s_{1,2}}a_{k_2,s_{1,2}}=e
		%\end{array} $\\
		%$\begin{array}{lllllll}
		%\tau(k_2s_{1,3}^2k_2^{-1})&=&\tau(s_{1,2}s_{1,3}s_{1,3}s_{1,2})=a_{k_1,s_{1,2}}a_{k_2,s_{1,3}}a_{k_5,s_{1,3}}a_{k_2,s_{1,2}}=e
		%\end{array} $\\
		%$\begin{array}{llllllll}
		%\tau(k_2s_{1,4}^2k_2^{-1})&=&\tau(s_{1,2}s_{1,4}s_{1,4}s_{1,2})=a_{k_1,s_{1,2}}a_{k_2,s_{1,4}}a_{k_6,s_{1,4}}a_{k_2,s_{1,2}}=e
		%\end{array} $\\
		
		$\begin{array}{llllll}
		\tau(k_2(s_{1,2}s_{1,3}s_{1,4}s_{1,3})^2 k_2^{-1})=\tau(s_{1,2}s_{1,2}s_{1,3}s_{1,4}s_{1,3}s_{1,2}s_{1,3}s_{1,4}s_{1,3}s_{1,2})\\
		=a_{k_1,s_{1,2}}a_{k_2,s_{1,2}}a_{k_1,s_{1,3}}a_{k_3,s_{1,4}}a_{k_8,s_{1,3}}a_{k_{17},s_{1,2}}a_{k_{19},s_{1,3}}a_{k_{12},s_{1,4}}a_{k_5,s_{1,3}}a_{k_2,s_{1,2}}=
		a_{k_{17},s_{1,2}}a_{k_{19},s_{1,3}}
		\end{array} $\\
		
		$\begin{array}{lllll}
		\tau(k_2(s_{1,2}s_{1,4})^4k_2^{-1})=\tau(s_{1,2}s_{1,2}s_{1,4}s_{1,2}s_{1,4}s_{1,2}s_{1,4}s_{1,2}s_{1,4}s_{1,2})\\
		=a_{k_1,s_{1,2}}a_{k_2,s_{1,2}}a_{k_1,s_{1,4}}a_{k_4,s_{1,2}}a_{k_9,s_{1,4}}a_{k_{19},s_{1,2}}a_{k_{17}s_{1,4}}a_{k_{13},s_{1,2}}a_{k_6,s_{1,4}}a_{k_2,s_{1,2}}=
		a_{k_{19},s_{1,2}}a_{k_{17}s_{1,4}}
		\end{array} $\\
		
		%$\begin{array}{llllll}
		%\tau(k_3s_{1,2}^2k_3^{-1})&=&\tau(s_{1,3}s_{1,2}s_{1,2}s_{1,3})=a_{k_1,s_{1,3}}a_{k_3,s_{1,2}}a_{k_7,s_{1,2}}a_{k_3,s_{1,3}}=e
		%\end{array} $\\
		%$\begin{array}{llllll}
		%\tau(k_3s_{1,3}^2k_3^{-1})&=&\tau(s_{1,3}s_{1,3}s_{1,3}s_{1,3})=a_{k_1,s_{1,3}}a_{k_3,s_{1,3}}a_{k_1,s_{1,3}}a_{k_3,s_{1,3}}=e
		%\end{array} $\\
		%$\begin{array}{lllllll}
		%\tau(k_3s_{1,4}^2k_3^{-1})&=& \tau(s_{1,3}s_{1,4}s_{1,4}s_{1,3}=a_{k_1,s_{1,3}}a_{k_3,s_{1,4}}a_{k_8,s_{1,4}}a_{k_3,s_{1,3}}=e
		%\end{array}$\\
		
		$\begin{array}{lllllll}
		\tau(k_3(s_{1,2}s_{1,3}s_{1,4}s_{1,3})^2 k_3^{-1})=\tau(s_{1,3}s_{1,2}s_{1,3}s_{1,4}s_{1,3}s_{1,2}s_{1,3}s_{1,4}s_{1,3}s_{1,3})\\
		=a_{k_1,s_{1,3}}a_{k_3,s_{1,2}}a_{k_7,s_{1,3}}a_{k_{11},s_{1,4}}a_{k_{22},s_{1,3}}a_{k_{20},s_{1,2}}a_{k_{10},s_{1,3}}a_{k_4,s_{1,4}}a_{k_1,s_{1,3}}a_{k_3,s_{1,3}}=
		a_{k_7,s_{1,3}} a_{k_{22},s_{1,3}}
		\end{array} $\\
		
		$\begin{array}{lllll}
		\tau(k_3(s_{1,2}s_{1,4})^4k_3^{-1})=\tau(s_{1,3}s_{1,2}s_{1,4}s_{1,2}s_{1,4}s_{1,2}s_{1,4}s_{1,2}s_{1,4}s_{1,3})\\
		=a_{k_1,s_{1,3}}a_{k_3,s_{1,2}}a_{k_7,s_{1,4}}a_{k_{15},s_{1,2}}a_{k_{21},s_{1,4}}a_{k_{10},s_{1,2}}a_{k_{20},s_{1,4}}a_{k_{16},s_{1,2}}a_{k_8,s_{1,4}}a_{k_3,s_{1,3}}=
		a_{k_{15},s_{1,2}} a_{k_{20},s_{1,4}}
		\end{array} $\\
		
		%$\begin{array}{llllll}
		%\tau(k_4s_{1,2}^2k_4^{-1})&=&\tau(s_{1,4}s_{1,2}s_{1,2}s_{1,4})=a_{k_1,s_{1,4}}a_{k_4,s_{1,2}}a_{k_9,s_{1,2}}a_{k_4,s_{1,4}}=e
		%\end{array} $\\
		%$\begin{array}{lllllll}
		%\tau(k_4s_{1,3}^2k_4^{-1})&=&\tau(s_{1,4}s_{1,3}s_{1,3}s_{1,4})=a_{k_1,s_{1,4}}a_{k_4,s_{1,3}}a_{k_{10},s_{1,3}}a_{k_4,s_{1,4}}=e
		%\end{array} $\\
		%$\begin{array}{lllll}
		%\tau(k_4s_{1,4}^2k_4^{-1})&=&\tau(s_{1,4}s_{1,4}s_{1,4}s_{1,4})=a_{k_1,s_{1,4}}a_{k_4,s_{1,4}}a_{k_{1},s_{1,4}}a_{k_4,s_{1,4}}=e
		%\end{array} $\\
		
		$\begin{array}{llllll}
		\tau(k_4(s_{1,2}s_{1,3}s_{1,4}s_{1,3})^2 k_4^{-1})=\tau(s_{1,4}s_{1,2}s_{1,3}s_{1,4}s_{1,3}s_{1,2}s_{1,3}s_{1,4}s_{1,3}s_{1,4})\\
		=a_{k_1,s_{1,4}}a_{k_4,s_{1,2}}a_{k_9,s_{1,3}}a_{k_{18},s_{1,4}}a_{k_{14},s_{1,3}}a_{k_6,s_{1,2}}a_{k_{13},s_{1,3}}a_{k_{21},s_{1,4}}a_{k_{10},s_{1,3}}a_{k_4,s_{1,4}}=
		a_{k_{18},s_{1,4}} a_{k_{13},s_{1,3}}
		\end{array} $\\
		
		$\begin{array}{llllllllll}
		\tau(k_4(s_{1,2}s_{1,4})^4k_4^{-1})=\tau(s_{1,4}s_{1,2}s_{1,4}s_{1,2}s_{1,4}s_{1,2}s_{1,4}s_{1,2}s_{1,4}s_{1,4})\\
		=a_{k_1,s_{1,4}}a_{k_4,s_{1,2}}a_{k_9,s_{1,4}}a_{k_{19},s_{1,2}}a_{k_{17},s_{1,4}}a_{k_{13},s_{1,2}}a_{k_6,s_{1,4}}a_{k_2,s_{1,2}}a_{k_1,s_{1,4}}a_{k_4,s_{1,4}}=
		a_{k_{19},s_{1,2}}a_{k_{17},s_{1,4}}
		\end{array} $\\
		
		%$\begin{array}{lllll}
		%\tau(k_5s_{1,2}^2k_5^{-1})&=&\tau(s_{1,2}s_{1,3}s_{1,2}s_{1,2}s_{1,3}s_{1,2})
		%=a_{k_1,s_{1,2}}a_{k_2,s_{1,3}}a_{k_5,s_{1,2}}a_{k_{11},s_{1,2}}a_{k_5,s_{1,3}}a_{k_2,s_{1,2}}=e
		%\end{array} $\\
		%$\begin{array}{llllll}
		%\tau(k_5s_{1,3}^2k_5^{-1})&=&\tau(s_{1,2}s_{1,3}s_{1,3}s_{1,3}s_{1,3}s_{1,2})=
		%a_{k_1,s_{1,2}}a_{k_2,s_{1,3}}a_{k_5,s_{1,3}}a_{k_2,s_{1,3}}a_{k_5,s_{1,3}}a_{k_2,s_{1,2}}=e
		%\end{array} $\\
		%$\begin{array}{llllll}
		%\tau(k_5s_{1,4}^2k_5^{-1})=\tau(s_{1,2}s_{1,3}s_{1,4}s_{1,4}s_{1,3}s_{1,2})
		%=a_{k_1,s_{1,2}}a_{k_2,s_{1,3}}a_{k_5,s_{1,4}}a_{k_{12},s_{1,4}}a_{k_5,s_{1,3}}a_{k_2,s_{1,2}}=e
		%\end{array} $\\
		
		$\begin{array}{lllllll}
		\tau(k_5(s_{1,2}s_{1,3}s_{1,4}s_{1,3})^2 k_5^{-1})=\tau(s_{1,2}s_{1,3}s_{1,2}s_{1,3}s_{1,4}s_{1,3}s_{1,2}s_{1,3}s_{1,4}s_{1,3}s_{1,3}s_{1,2})\\
		=a_{k_1,s_{1,2}}a_{k_2,s_{1,3}}a_{k_5,s_{1,2}}a_{k_{11},s_{1,3}}a_{k_7,s_{1,4}}a_{k_{15},s_{1,3}}a_{k_{24},s_{1,2}}a_{k_{14},s_{1,3}}a_{k_6,s_{1,4}}a_{k_2,s_{1,3}} a_{k_5,s_{1,3}}a_{k_2,s_{1,2}}=\\
		a_{k_{11},s_{1,3}} a_{k_{15},s_{1,3}}
		\end{array} $\\
		
		$\begin{array}{lllllll}
		\tau(k_5(s_{1,2}s_{1,4})^4k_5^{-1})=\tau(s_{1,2}s_{1,3}s_{1,2}s_{1,4}s_{1,2}s_{1,4}s_{1,2}s_{1,4}s_{1,2}s_{1,4}s_{1,3}s_{1,2})\\
		=a_{k_1,s_{1,2}}a_{k_2,s_{1,3}}a_{k_5,s_{1,2}}a_{k_{11},s_{1,4}}a_{k_{22},s_{1,2}}a_{k_{18},s_{1,4}}a_{k_{14},s_{1,2}}a_{k_{24},s_{1,4}}a_{k_{23},s_{1,2}}a_{k_{12},s_{1,4}}a_{k_5,s_{1,3}} a_{k_2,s_{1,2}}=\\
		a_{k_{22},s_{1,2}}a_{k_{18},s_{1,4}} a_{k_{24},s_{1,4}}
		\end{array} $\\
		
		%$\begin{array}{lllllll}
		%\tau(k_6s_{1,2}^2k_6^{-1})=\tau(s_{1,2}s_{1,4}s_{1,2}s_{1,2}s_{1,4}s_{1,2)}=
		%a_{k_1,s_{1,2}}a_{k_2,s_{1,4}}a_{k_6,s_{1,2}}a_{k_{13},s_{1,2}}a_{k_6,s_{1,4}}a_{k_2,s_{1,2}}=e
		%\end{array} $\\
		%$\begin{array}{llllll}
		%\tau(k_6s_{1,3}^2k_6^{-1})=\tau(s_{1,2}s_{1,4}s_{1,3}s_{1,3}s_{1,4}s_{1,2})=
		%a_{k_1,s_{1,2}}a_{k_2,s_{1,4}}a_{k_6,s_{1,3}}a_{k_{14},s_{1,3}}a_{k_6,s_{1,4}}a_{k_2,s_{1,2}}=e
		%\end{array} $\\
		%$\begin{array}{lllll}
		%%\tau(k_6s_{1,4}^2k_6^{-1})=\tau(s_{1,2}s_{1,4}s_{1,4}s_{1,4}s_{1,4}s_{1,2})=
		%a_{k_1,s_{1,2}}a_{k_2,s_{1,4}}a_{k_6,s_{1,4}}a_{k_{2},s_{1,4}}a_{k_6,s_{1,4}}a_{k_2,s_{1,2}}=e
		%\end{array} $\\
		
		$\begin{array}{lllllll}
		\tau(k_6(s_{1,2}s_{1,3}s_{1,4}s_{1,3})^2 k_6^{-1})=\tau(s_{1,2}s_{1,4}s_{1,2}s_{1,3}s_{1,4}s_{1,3}s_{1,2}s_{1,3}s_{1,4}s_{1,3}s_{1,4}s_{1,2})\\
		=a_{k_1,s_{1,2}}a_{k_2,s_{1,4}}a_{k_6,s_{1,2}}a_{k_{13},s_{1,3}}a_{k_{21},s_{1,4}}a_{k_{10},s_{1,3}}
		a_{k_4,s_{1,2}}a_{k_9,s_{1,3}}a_{k_{18},s_{1,4}}a_{k_{14},s_{1,3}}a_{k_{6},s_{1,4}}a_{k_2,s_{1,2}}=\\
		a_{k_{13},s_{1,3}} a_{k_{18},s_{1,4}}
		\end{array} $\\
		
		$\begin{array}{llllllll}
		\tau(k_6(s_{1,2}s_{1,4})^4k_6^{-1})=
		\tau(s_{1,2}s_{1,4}s_{1,2}s_{1,4}s_{1,2}s_{1,4}s_{1,2}s_{1,4}s_{1,2}s_{1,4}s_{1,4}s_{1,2})\\
		=a_{k_1,s_{1,2}}a_{k_2,s_{1,4}}a_{k_{6},s_{1,2}}a_{k_{13},s_{1,4}}a_{k_{17},s_{1,2}}a_{k_{19},s_{1,4}}
		a_{k_{9},s_{1,2}}a_{k_4,s_{1,4}}a_{k_1,s_{1,2}}a_{k_2,s_{1,4}}a_{k_6,s_{1,4}}a_{k_2,s_{1,2}}=\\
		a_{k_{13},s_{1,4}} a_{k_{17},s_{1,2}}
		\end{array} $\\
		
		%$\begin{array}{llllll}
		%\tau(k_7s_{1,2}^2k_7^{-1})=\tau(s_{1,3}s_{1,2}s_{1,2}s_{1,2}s_{1,2}s_{1,3})=
		%a_{k_1,s_{1,3}}a_{k_3,s_{1,2}}a_{k_7,s_{1,2}}a_{k_3,s_{1,2}}a_{k_7,s_{1,2}}a_{k_3,s_{1,3}}=e
		%\end{array} $\\
		%$\begin{array}{llllll}
		%\tau(k_7s_{1,3}^2k_7^{-1})=\tau(s_{1,3}s_{1,2}s_{1,3}s_{1,3}s_{1,2}s_{1,3})=
		%a_{k_1,s_{1,3}}a_{k_3,s_{1,2}}a_{k_7,s_{1,3}}a_{k_{11},s_{1,3}}a_{k_7,s_{1,2}}a_{k_3,s_{1,3}}=a_{k_7,s_{1,3}}a_{k_{11},s_{1,3}}
		%\end{array} $\\
		%$\begin{array}{lllll}
		%\tau(k_7s_{1,4}^2k_7^{-1})=\tau(s_{1,3}s_{1,2}s_{1,4}s_{1,4}s_{1,2}s_{1,3})=
		%a_{k_1,s_{1,3}}a_{k_3,s_{1,2}}a_{k_7,s_{1,4}}a_{k_{15},s_{1,4}}a_{k_7,s_{1,2}}a_{k_3,s_{1,3}}=e
		%\end{array} $\\
		
		$\begin{array}{llllllll}
		\tau(k_7(s_{1,2}s_{1,3}s_{1,4}s_{1,3})^2 k_7^{-1})=\tau(s_{1,3}s_{1,2}s_{1,2}s_{1,3}s_{1,4}s_{1,3}s_{1,2}s_{1,3}s_{1,4}s_{1,3}s_{1,2}s_{1,3})\\
		=a_{k_1,s_{1,3}}a_{k_3,s_{1,2}}a_{k_7,s_{1,2}}a_{k_3,s_{1,3}}a_{k_1,s_{1,4}}a_{k_4,s_{1,3}}a_{k_{10},s_{1,2}}
		a_{k_{20},s_{1,3}}a_{k_{22},s_{1,4}}a_{k_{11},s_{1,3}}a_{k_7,s_{1,2}}a_{k_3,s_{1,3}}=\\
		a_{k_{20},s_{1,3}}a_{k_{11},s_{1,3}}
		\end{array} $\\
		
		$\begin{array}{lllll}
		\tau(k_7(s_{1,2}s_{1,4})^4k_7^{-1})=\tau(s_{1,3}s_{1,2}s_{1,2}s_{1,4}s_{1,2}s_{1,4}s_{1,2}s_{1,4}s_{1,2}s_{1,4}s_{1,2}s_{1,3})\\
		=a_{k_1,s_{1,3}}a_{k_3,s_{1,2}}a_{k_7,s_{1,2}}a_{k_3,s_{1,4}}a_{k_8,s_{1,2}}a_{k_{16},s_{1,4}}a_{k_{20},s_{1,2}}a_{k_{10},s_{1,4}}
		a_{k_{21},s_{1,2}}a_{k_{15},s_{1,4}} a_{k_7,s_{1,2}}a_{k_3,s_{1,3}}=\\
		a_{k_{16},s_{1,4}}a_{k_{21},s_{1,2}}
		\end{array} $\\
		
		%$\begin{array}{lllllll}
		%\tau(k_8s_{1,2}^2k_8^{-1})=\tau(s_{1,3}s_{1,4}s_{1,2}s_{1,2}s_{1,4}s_{1,3})=
		%a_{k_1,s_{1,3}}a_{k_3,s_{1,4}}a_{k_8,s_{1,2}}a_{k_{16},s_{1,2}}a_{k_8,s_{1,4}}a_{k_3,s_{1,3}}=e
		%\end{array} $\\
		%$\begin{array}{lllllll}
		%\tau(k_8s_{1,3}^2k_8^{-1})=\tau(s_{1,3}s_{1,4}s_{1,3}s_{1,3}s_{1,4}s_{1,3})=
		%a_{k_1,s_{1,3}}a_{k_3,s_{1,4}}a_{k_8,s_{1,3}}a_{k_{17},s_{1,3}}a_{k_8,s_{1,4}}a_{k_3,s_{1,3}}=e
		%\end{array} $\\
		%$\begin{array}{lllll}
		%\tau(k_8s_{1,4}^2k_8^{-1})=\tau(s_{1,3}s_{1,4}s_{1,4}s_{1,4}s_{1,4}s_{1,3})=
		%a_{k_1,s_{1,3}}a_{k_3,s_{1,4}}a_{k_8,s_{1,4}}a_{k_3,s_{1,4}}a_{k_8,s_{1,4}}a_{k_3,s_{1,3}}=e
		%\end{array} $\\

		$\begin{array}{llllll}
		\tau(k_8(s_{1,2}s_{1,3}s_{1,4}s_{1,3})^2 k_8^{-1})=\tau(s_{1,3}s_{1,4}s_{1,2}s_{1,3}s_{1,4}s_{1,3}s_{1,2}s_{1,3}s_{1,4}s_{1,3}s_{1,4}s_{1,3})\\
		=a_{k_1,s_{1,3}}a_{k_3,s_{1,4}}a_{k_8,s_{1,2}}a_{k_{16},s_{1,3}}a_{k_{23},s_{1,4}}
		a_{k_{24},s_{1,3}}a_{k_{15},s_{1,2}}a_{k_{21},s_{1,3}}a_{k_{13},s_{1,4}}a_{k_{17},s_{1,3}}a_{k_8,s_{1,4}}a_{k_3,s_{1,3}}=\\
		a_{k_{16},s_{1,3}} a_{k_{23},s_{1,4}} a_{k_{24},s_{1,3}} a_{k_{15},s_{1,2}} a_{k_{21},s_{1,3}} a_{k_{13},s_{1,4}}
		\end{array} $\\
		
		$\begin{array}{lllllll}
		\tau(k_8(s_{1,2}s_{1,4})^4k_8^{-1})=
		\tau(s_{1,3}s_{1,4}s_{1,2}s_{1,4}s_{1,2}s_{1,4}s_{1,2}s_{1,4}s_{1,2}s_{1,4}s_{1,4}s_{1,3})\\
		=a_{k_1,s_{1,3}}a_{k_3,s_{1,4}}a_{k_8,s_{1,2}}a_{k_{16},s_{1,4}}a_{k_{20},s_{1,2}}a_{k_{10},s_{1,4}}a_{k_{21},s_{1,2}}
		a_{k_{15},s_{1,4}}a_{k_7,s_{1,2}}a_{k_3,s_{1,4}}a_{k_8,s_{1,4}}a_{k_3,s_{1,3}}=\\
		a_{k_{16},s_{1,4}}a_{k_{21},s_{1,2}}
		\end{array} $\\
		
		%$\begin{array}{lllll}
		%\tau(k_9s_{1,2}^2k_9^{-1})=\tau(s_{1,4}s_{1,2}s_{1,2}s_{1,2}s_{1,2}s_{1,4})=
		%a_{k_1,s_{1,4}}a_{k_4,s_{1,2}}a_{k_9,s_{1,2}}a_{k_4,s_{1,2}}a_{k_9,s_{1,2}}a_{k_4,s_{1,4}}=e
		%\end{array} $\\
		%$\begin{array}{lllllll}
		%\tau(k_9s_{1,3}^2k_9^{-1})=\tau(s_{1,4}s_{1,2}s_{1,3}s_{1,3}s_{1,2}s_{1,4})=
		%a_{k_1,s_{1,4}}a_{k_4,s_{1,2}}a_{k_9,s_{1,3}}a_{k_{18},s_{1,3}}a_{k_9,s_{1,2}}a_{k_4,s_{1,4}}=e
		%\end{array} $\\
		%$\begin{array}{lllllll}
		%\tau(k_9s_{1,4}^2k_9^{-1})=\tau(s_{1,4}s_{1,2}s_{1,4}s_{1,4}s_{1,2}s_{1,4})=
		%a_{k_1,s_{1,4}}a_{k_4,s_{1,2}}a_{k_9,s_{1,4}}a_{k_{19},s_{1,4}}a_{k_9,s_{1,2}}a_{k_4,s_{1,4}}=e
		%\end{array} $\\
		
		$\begin{array}{llllll}
		\tau(k_9(s_{1,2}s_{1,3}s_{1,4}s_{1,3})^2 k_9^{-1})=\tau(s_{1,4}s_{1,2}s_{1,2}s_{1,3}s_{1,4}s_{1,3}s_{1,2}s_{1,3}s_{1,4}s_{1,3}s_{1,2}s_{1,4})\\
		=a_{k_1,s_{1,4}}a_{k_4,s_{1,2}}a_{k_9,s_{1,2}}a_{k_4,s_{1,3}}a_{k_{10},s_{1,4}}a_{k_{21},s_{1,3}}a_{k_{13},s_{1,2}}
		a_{k_6,s_{1,3}}a_{k_{14},s_{1,4}}a_{k_{18},s_{1,3}} a_{k_9,s_{1,2}}a_{k_4,s_{1,4}}=\\
		a_{k_{21},s_{1,3}}a_{k_{14},s_{1,4}}
		\end{array} $\\
		
		$\begin{array}{lllll}
		\tau(k_9(s_{1,2}s_{1,4})^4k_9^{-1})=
		\tau(s_{1,4}s_{1,2}s_{1,2}s_{1,4}s_{1,2}s_{1,4}s_{1,2}s_{1,4}s_{1,2}s_{1,4}s_{1,2}s_{1,4})\\
		=a_{k_1,s_{1,4}}a_{k_4,s_{1,2}}a_{k_9,s_{1,2}}a_{k_4,s_{1,4}}a_{k_1,s_{1,2}}
		a_{k_2,s_{1,4}}a_{k_6,s_{1,2}}a_{k_{13},s_{1,4}}a_{k_{17},s_{1,2}}a_{k_{19},s_{1,4}}a_{k_9,s_{1,2}}a_{k_4,s_{1,4}}=\\
		a_{k_{13},s_{1,4}}a_{k_{17},s_{1,2}}
		\end{array} $\\
		
		%$\begin{array}{llllll}
		%\tau(k_{10}s_{1,2}^2k_{10}^{-1})=\tau(s_{1,4}s_{1,3}s_{1,2}s_{1,2}s_{1,3}s_{1,4})=
		%a_{k_1,s_{1,4}}a_{k_4,s_{1,3}}a_{k_{10},s_{1,2}}a_{k_{20},s_{1,2}}a_{k_{10},s_{1,3}}a_{k_4,s_{1,4}}=e
		%\end{array} $\\
		%$\begin{array}{lllllll}
		%\tau(k_{10}s_{1,3}^2k_{10}^{-1})=\tau(s_{1,4}s_{1,3}s_{1,3}s_{1,3}s_{1,3}s_{1,4})=
		%a_{k_1,s_{1,4}}a_{k_4,s_{1,3}}a_{k_{10},s_{1,3}}a_{k_4,s_{1,3}}a_{k_{10},s_{1,3}}a_{k_4,s_{1,4}}=e
		%\end{array} $\\
		%$\begin{array}{llllll}
		%\tau(k_{10}s_{1,4}^2k_{10}^{-1})=\tau(s_{1,4}s_{1,3}s_{1,4}s_{1,4}s_{1,3}s_{1,4})=
		%a_{k_1,s_{1,4}}a_{k_4,s_{1,3}}a_{k_{10},s_{1,4}}a_{k_{21},s_{1,4}}a_{k_{10},s_{1,3}}a_{k_4,s_{1,4}}=e
		%\end{array} $\\
		
		$\begin{array}{llllll}
		\tau(k_{10}(s_{1,2}s_{1,3}s_{1,4}s_{1,3})^2 k_{10}^{-1})=\tau(s_{1,4}s_{1,3}s_{1,2}s_{1,3}s_{1,4}s_{1,3}s_{1,2}s_{1,3}s_{1,4}s_{1,3}s_{1,3}s_{1,4})=\\
		a_{k_1,s_{1,4}}a_{k_4,s_{1,3}}a_{k_{10},s_{1,2}}a_{k_{20},s_{1,3}}a_{k_{22},s_{1,4}}a_{k_{11},s_{1,3}}a_{k_{7},s_{1,2}}
		a_{k_3,s_{1,3}}a_{k_1,s_{1,4}}a_{k_4,s_{1,3}}a_{k_{10},s_{1,3}}a_{k_4,s_{1,4}}=
		a_{k_{20},s_{1,3}}a_{k_{11},s_{1,3}}
		\end{array} $\\
		
		$\begin{array}{llllll}
		\tau(k_{10}(s_{1,2}s_{1,4})^4k_{10}^{-1})=\tau(s_{1,4}s_{1,3}
		s_{1,2}s_{1,4}s_{1,2}s_{1,4}s_{1,2}s_{1,4}s_{1,2}s_{1,4}s_{1,3}s_{1,4})\\
		=a_{k_1,s_{1,4}}a_{k_4,s_{1,3}}a_{k_{10},s_{1,2}}a_{k_{20},s_{1,4}}a_{k_{16},s_{1,2}}a_{k_8,s_{1,4}}a_{k_3,s_{1,2}}a_{k_7,s_{1,4}}a_{k_{15},s_{1,2}}a_{k_{21},s_{1,4}}a_{k_{10},s_{1,3}}a_{k_4,s_{1,4}}=\\
		a_{k_{20},s_{1,4}}a_{k_{15},s_{1,2}}
		\end{array} $\\
		
		%$\begin{array}{lllllllll}
		%\tau(k_{11}s_{1,2}^2k_{11}^{-1})=\tau(s_{1,2}s_{1,3}s_{1,2}s_{1,2}s_{1,2}s_{1,2}s_{1,3}s_{1,2})
		%=a_{k_1,s_{1,2}}a_{k_2,s_{1,3}}a_{k_5,s_{1,2}}a_{k_{11},s_{1,2}}a_{k_5,s_{1,2}}a_{k_{11},s_{1,2}}a_{k_5,s_{1,3}}a_{k_2,s_{1,2}}=e
		%\end{array} $\\
		$\begin{array}{llllll}
		\tau(k_{11}s_{1,3}^2k_{11}^{-1})=\tau(s_{1,2}s_{1,3}s_{1,2}s_{1,3}s_{1,3}s_{1,2}s_{1,3}s_{1,2})
		=a_{k_1,s_{1,2}}a_{k_2,s_{1,3}}a_{k_5,s_{1,2}}a_{k_{11},s_{1,3}}a_{k_7,s_{1,3}}a_{k_{11},s_{1,2}}\\
		a_{k_5,s_{1,3}}a_{k_2,s_{1,2}}=
		a_{k_{11},s_{1,3}}a_{k_7,s_{1,3}}
		\end{array} $\\
		%$\begin{array}{lllllll}
		%\tau(k_{11}s_{1,4}^2k_{11}^{-1})=\tau(s_{1,2}s_{1,3}s_{1,2}s_{1,4}s_{1,4}s_{1,2}s_{1,3}s_{1,2})=
		%a_{k_1,s_{1,2}}a_{k_2,s_{1,3}}a_{k_5,s_{1,2}}a_{k_{11},s_{1,4}}a_{k_{22},s_{1,4}}a_{k_{11},s_{1,2}}a_{k_5,s_{1,3}}a_{k_2,s_{1,2}}=e
		%\end{array} $\\
		
		$\begin{array}{lllllll}
		\tau(k_{11}(s_{1,2}s_{1,3}s_{1,4}s_{1,3})^2 k_{11}^{-1})=\tau(s_{1,2}s_{1,3}s_{1,2}s_{1,2}s_{1,3}s_{1,4}s_{1,3}s_{1,2}s_{1,3}s_{1,4}s_{1,3}s_{1,2}s_{1,3}s_{1,2})\\
		=a_{k_1,s_{1,2}}a_{k_2,s_{1,3}}a_{k_5,s_{1,2}}
		a_{k_{11},s_{1,2}}a_{k_5,s_{1,3}}a_{k_2,s_{1,4}}
		a_{k_6,s_{1,3}}a_{k_{14},s_{1,2}}a_{k_{24},s_{1,3}}a_{k_{15},s_{1,4}}a_{k_7,s_{1,3}}a_{k_{11},s_{1,2}}a_{k_5,s_{1,3}}a_{k_2,s_{1,2}}=\\
		a_{k_{24},s_{1,3}}a_{k_7,s_{1,3}}
		\end{array} $\\
		
		$\begin{array}{llllllll}
		\tau(k_{11}(s_{1,2}s_{1,4})^4k_{11}^{-1})=\tau(s_{1,2}s_{1,3}s_{1,2}s_{1,2}s_{1,4}s_{1,2}s_{1,4}s_{1,2}s_{1,4}s_{1,2}s_{1,4}s_{1,2}s_{1,3}s_{1,2})\\
		=a_{k_1,s_{1,2}}a_{k_2,s_{1,3}}a_{k_5,s_{,2}}a_{k_{11},s_{1,2}}a_{k_5,s_{1,4}}
		a_{k_{12},s_{1,2}}a_{k_{23},s_{1,4}}a_{k_{24},s_{1,2}}a_{k_{14},s_{1,4}}a_{k_{18},s_{1,2}}a_{k_{22},s_{1,4}}
		a_{k_{11},s_{1,2}}a_{k_5,s_{1,3}}a_{k_2,s_{1,2}}=\\
		a_{k_{23},s_{1,4}}a_{k_{18},s_{1,2}}
		\end{array} $\\
		
		%$\begin{array}{lllllll}
		%\tau(k_{12}s_{1,2}^2k_{12}^{-1})=\tau(s_{1,2}s_{1,3}s_{1,4}s_{1,2}s_{1,2}s_{1,4}s_{1,3}s_{1,2})
		%=a_{k_1,s_{1,2}}a_{k_2,s_{1,3}}a_{k_5,s_{1,4}}a_{k_{12},s_{1,2}}a_{k_{23},s_{1,2}}a_{k_{12},s_{1,4}}a_{k_5,s_{1,3}}a_{k_2,s_{1,2}}=e
		%\end{array} $\\
		$\begin{array}{llllll}
		\tau(k_{12}s_{1,3}^2k_{12}^{-1})=\tau(s_{1,2}s_{1,3}s_{1,4}s_{1,3}s_{1,3}s_{1,4}s_{1,3}s_{1,2})=
		a_{k_1,s_{1,2}}a_{k_2,s_{1,3}}a_{k_5,s_{1,4}}a_{k_{12},s_{1,3}}a_{k_{19},s_{1,3}}a_{k_{12},s_{1,4}}\\
		a_{k_5,s_{1,3}}a_{k_2,s_{1,2}}=
		a_{k_{12},s_{1,3}}a_{k_{19},s_{1,3}}
		\end{array} $\\
		%$\begin{array}{lllllll}
		%\tau(k_{12}s_{1,4}^2k_{12}^{-1})=\tau(s_{1,2}s_{1,3}s_{1,4}s_{1,4}s_{1,4}s_{1,4}s_{1,3}s_{1,2})=
		%a_{k_1,s_{1,2}}a_{k_2,s_{1,3}}a_{k_5,s_{1,4}}a_{k_{12},s_{1,4}}a_{k_5,s_{1,4}}a_{k_{12},s_{1,4}}a_{k_5,s_{1,3}}a_{k_2,s_{1,2}}=e
		%\end{array} $\\
		
		$\begin{array}{llllllll}
		\tau(k_{12}(s_{1,2}s_{1,3}s_{1,4}s_{1,3})^2 k_{12}^{-1})=\tau(s_{1,2}s_{1,3}s_{1,4}s_{1,2}s_{1,3}s_{1,4}s_{1,3}
		s_{1,2}s_{1,3}s_{1,4}s_{1,3}s_{1,4}s_{1,3}s_{1,2})\\
		=a_{k_1,s_{1,2}}a_{k_2,s_{1,3}}a_{k_5,s_{1,4}}a_{k_{12},s_{1,2}}a_{k_{23},s_{1,3}}a_{k_{16},s_{1,4}}
		a_{k_{20},s_{1,3}}a_{k_{22},s_{1,2}}a_{k_{18},s_{1,3}}a_{k_9,s_{1,4}}a_{k_{19},s_{1,3}}a_{k_{12},s_{1,4}}a_{k_5,s_{1,3}}a_{k_2,s_{1,2}}=\\
		a_{k_{23},s_{1,3}}a_{k_{16},s_{1,4}}
		a_{k_{20},s_{1,3}}a_{k_{22},s_{1,2}} a_{k_{19},s_{1,3}}
		\end{array} $\\
		
		$\begin{array}{llllllll}
		\tau(k_{12}(s_{1,2}s_{1,4})^4k_{12}^{-1})=\tau(s_{1,2}s_{1,3}s_{1,4}s_{1,2}s_{1,4}s_{1,2}s_{1,4}s_{1,2}s_{1,4}s_{1,2}s_{1,4}s_{1,4}s_{1,3}s_{1,2})\\
		=a_{k_1,s_{1,2}}a_{k_2,s_{1,3}}a_{k_5,s_{1,4}}a_{k_{12},s_{1,2}}a_{k_{23},s_{1,4}}a_{k_{24},s_{1,2}}a_{k_{14},s_{1,4}}
		a_{k_{18},s_{1,2}}a_{k_{22},s_{1,4}}a_{k_{11},s_{1,2}}a_{k_5,s_{1,4}}a_{k_{12},s_{1,4}}a_{k_5,s_{1,3}}a_{k_2,s_{1,2}}=\\
		a_{k_{23},s_{1,4}}a_{k_{14},s_{1,4}}a_{k_{18},s_{1,2}}
		\end{array} $\\
		
		%$\begin{array}{lllllllll}
		%\tau(k_{13}s_{1,2}^2k_{13}^{-1})=\tau(s_{1,2}s_{1,4}s_{1,2}s_{1,2}s_{1,2}s_{1,2}s_{1,4}s_{1,2})=
		%a_{k_1,s_{1,2}}a_{k_2,s_{1,4}}a_{k_6,s_{1,2}}a_{k_{13},s_{1,2}}a_{k_6,s_{1,2}}a_{k_{13},s_{1,2}}a_{k_6,s_{1,4}}a_{k_2,s_{1,2}}=e
		%\end{array} $\\
		$\begin{array}{llllll}
		\tau(k_{13}s_{1,3}^2k_{13}^{-1})=\tau(s_{1,2}s_{1,4}s_{1,2}s_{1,3}s_{1,3}s_{1,2}s_{1,4}s_{1,2})=
		a_{k_1,s_{1,2}}a_{k_2,s_{1,4}}a_{k_6,s_{1,2}}a_{k_{13},s_{1,3}}a_{k_{21},s_{1,3}}a_{k_{13},s_{1,2}}\\
		a_{k_6,s_{1,4}}a_{k_2,s_{1,2}}=
		a_{k_{13},s_{1,3}}a_{k_{21},s_{1,3}}
		\end{array} $\\
		
		$\begin{array}{lllllll}
		\tau(k_{13}s_{1,4}^2k_{13}^{-1})=\tau(s_{1,2}s_{1,4}s_{1,2}s_{1,4}s_{1,4}s_{1,2}s_{1,4}s_{1,2})=
		a_{k_1,s_{1,2}}a_{k_2,s_{1,4}}a_{k_6,s_{1,2}}a_{k_{13},s_{1,4}}a_{k_{17},s_{1,4}}a_{k_{13},s_{1,2}}\\
		a_{k_6,s_{1,4}}a_{k_2,s_{1,2}}=
		a_{k_{13},s_{1,4}}a_{k_{17},s_{1,4}}
		\end{array} $\\
		
		$\begin{array}{lllllllll}
		\tau(k_{13}(s_{1,2}s_{1,3}s_{1,4}s_{1,3})^2 k_{13}^{-1})=\tau(s_{1,2}s_{1,4}s_{1,2}s_{1,2}s_{1,3}s_{1,4}s_{1,3}s_{1,2}s_{1,3}s_{1,4}s_{1,3}s_{1,2}s_{1,4}s_{1,2})\\
		=a_{k_1,s_{1,2}}a_{k_2,s_{1,4}}a_{k_6,s_{1,2}}a_{k_{13},s_{1,2}}
		a_{k_6,s_{1,3}}a_{k_{14},s_{1,4}}a_{k_{18},s_{1,3}}a_{k_9,s_{1,2}}a_{k_4,s_{1,3}}a_{k_{10},s_{1,4}}
		a_{k_{21},s_{1,3}}a_{k_{13},s_{1,2}}a_{k_6,s_{1,4}}a_{k_2,s_{1,2}}=\\
		a_{k_{14},s_{1,4}} a_{k_{21},s_{1,3}}
		\end{array} $\\
		
		$\begin{array}{lllllllll}
		\tau(k_{13}(s_{1,2}s_{1,4})^4k_{13}^{-1})=\tau(s_{1,2}s_{1,4}s_{1,2}s_{1,2}s_{1,4}s_{1,2}s_{1,4}s_{1,2}s_{1,4}s_{1,2}s_{1,4}s_{1,2}s_{1,4}s_{1,2})\\
		=a_{k_1,s_{1,2}}a_{k_2,s_{1,4}}a_{k_6,s_{1,2}}a_{k_{13},s_{1,2}}
		a_{k_6,s_{1,4}}a_{k_2,s_{1,2}}a_{k_1,s_{1,4}}a_{k_4,s_{1,2}}a_{k_9,s_{1,4}}a_{k_{19},s_{1,2}}a_{k_{17},s_{1,4}}
		a_{k_{13},s_{1,2}}a_{k_6,s_{1,}}a_{k_2,s_{1,2}}=\\
		a_{k_{19},s_{1,2}}a_{k_{17},s_{1,4}}
		\end{array} $\\
		
		%$\begin{array}{lllll}
		%\tau(k_{14}s_{1,2}^2k_{14}^{-1})=\tau(s_{1,2}s_{1,4}s_{1,3}s_{1,2}s_{1,2}s_{1,3}s_{1,4}s_{1,2})=
		%a_{k_1,s_{1,2}}a_{k_2,s_{1,4}}a_{k_6,s_{1,3}}a_{k_{14},s_{1,2}}a_{k_{24},s_{1,2}}a_{k_{14},s_{1,3}}a_{k_6,s_{1,4}}a_{k_2,s_{1,2}}=e
		%\end{array} $\\
		%$\begin{array}{llllll}
		%\tau(k_{14}s_{1,3}^2k_{14}^{-1})=\tau(s_{1,2}s_{1,4}s_{1,3}s_{1,3}s_{1,3}s_{1,3}s_{1,4}s_{1,2})=
		%a_{k_1,s_{1,2}}a_{k_2,s_{1,4}}a_{k_6,s_{1,3}}a_{k_{14},s_{1,3}}a_{k_6,s_{1,3}}a_{k_{14},s_{1,3}}a_{k_6,s_{1,4}}a_{k_2,s_{1,2}}=e
		%\end{array} $\\
		$\begin{array}{lllll}
		\tau(k_{14}s_{1,4}^2k_{14}^{-1})\tau(s_{1,2}s_{1,4}s_{1,3}s_{1,4}s_{1,4}s_{1,3}s_{1,4}s_{1,2})=
		a_{k_1,s_{1,2}}a_{k_2,s_{1,4}}a_{k_6,s_{1,3}}a_{k_{14},s_{1,4}}a_{k_{18},s_{1,4}}a_{k_{14},s_{1,3}}a_{k_6,s_{1,4}}\\
		a_{k_2,s_{1,2}}=
		a_{k_{14},s_{1,4}}a_{k_{18},s_{1,4}}
		\end{array} $\\
		
		$\begin{array}{lllllll}
		\tau(k_{14}(s_{1,2}s_{1,3}s_{1,4}s_{1,3})^2 k_{14}^{-1})=\tau(s_{1,2}s_{1,4}s_{1,3}s_{1,2}s_{1,3}s_{1,4}s_{1,3}s_{1,2}s_{1,3}s_{1,4}s_{1,3}s_{1,3}s_{1,4}s_{1,2})\\
		=a_{k_1,s_{1,2}}a_{k_2,s_{1,4}}a_{k_6,s_{1,3}}a_{k_{14},s_{1,2}}a_{k_{24},s_{1,3}}a_{k_{15},s_{1,4}}a_{k_7,s_{1,3}}
		a_{k_{11},s_{1,2}}a_{k_5,s_{1,3}}a_{k_2,s_{1,4}}a_{k_6,s_{1,3}}a_{k_{14},s_{1,3}}a_{k_6,s_{1,4}}a_{k_2,s_{1,2}}
		=\\
		a_{k_{24},s_{1,3}} a_{k_7,s_{1,3}}
		\end{array} $\\
		
		$\begin{array}{lllllll}
		\tau(k_{14}(s_{1,2}s_{1,4})^4k_{14}^{-1})=\tau(s_{1,2}s_{1,4}s_{1,3}s_{1,2}s_{1,4}s_{1,2}s_{1,4}s_{1,2}s_{1,4}s_{1,2}s_{1,4}s_{1,3}s_{1,4}s_{1,2})\\
		=a_{k_1,s_{1,2}}a_{k_2,s_{1,4}}a_{k_6,s_{1,3}}a_{k_{14},s_{1,2}}a_{k_{24},s_{1,4}}
		a_{k_{23},s_{1,2}}a_{k_{12},s_{1,4}}a_{k_5,s_{1,2}}a_{k_{11},s_{1,4}}a_{k_{22},s_{1,2}}
		a_{k_{18},s_{1,4}} a_{k_{14},s_{1,3}}a_{k_6,s_{1,4}}a_{k_2,s_{1,2}}=\\
		a_{k_{24},s_{1,4}} a_{k_{22},s_{1,2}}
		a_{k_{18},s_{1,4}}
		\end{array} $\\
		
		$\begin{array}{lllllll}
		\tau(k_{15}s_{1,2}^2k_{15}^{-1})=\tau(s_{1,3}s_{1,2}s_{1,4}s_{1,2}s_{1,2}s_{1,4}s_{1,2}s_{1,3})=
		a_{k_1,s_{1,3}}a_{k_3,s_{1,2}}a_{k_7,s_{1,4}}a_{k_{15},s_{1,2}}a_{k_{21},s_{1,2}}a_{k_{15},s_{1,4}}\\
		a_{k_7,s_{1,2}}a_{k_3,s_{1,3}}=
		a_{k_{15},s_{1,2}}a_{k_{21},s_{1,2}}
		\end{array} $\\
		
		$\begin{array}{llllll}
		\tau(k_{15}s_{1,3}^2k_{15}^{-1})=\tau(s_{1,3}s_{1,2}s_{1,4}s_{1,3}s_{1,3}s_{1,4}s_{1,2}s_{1,3})
		=a_{k_1,s_{1,3}}a_{k_3,s_{1,2}}a_{k_7,s_{1,4}}a_{k_{15},s_{1,3}}a_{k_{24},s_{1,3}}a_{k_{15},s_{1,4}}\\
		a_{k_7,s_{1,2}}a_{k_3,s_{1,3}}=
		a_{k_{15},s_{1,3}}a_{k_{24},s_{1,3}}
		\end{array} $\\
		%$\begin{array}{lllllllll}
		%\tau(k_{15}s_{1,4}^2k_{15}^{-1})=\tau(s_{1,3}s_{1,2}s_{1,4}s_{1,3}s_{1,3}s_{1,4}s_{1,2}s_{1,3})=
		%a_{k_1,s_{1,3}}a_{k_3,s_{1,2}}a_{k_7,s_{1,4}}a_{k_{15},s_{1,4}}a_{k_7,s_{1,4}}a_{k_{15},s_{1,4}}a_{k_7,s_{1,2}}a_{k_3,s_{1,3}}=e
		%\end{array} $\\
		
		$\begin{array}{lllllllll}
		\tau(k_{15}(s_{1,2}s_{1,3}s_{1,4}s_{1,3})^2 k_{15}^{-1})=\tau(s_{1,3}s_{1,2}s_{1,4}s_{1,2}s_{1,3}s_{1,4}s_{1,3}s_{1,2}s_{1,3}s_{1,4}s_{1,3}s_{1,4}s_{1,2}s_{1,3})\\
		=a_{k_1,s_{1,3}}a_{k_3,s_{1,2}}a_{k_7,s_{1,4}}a_{k_{15},s_{1,2}}a_{k_{21},s_{1,3}}
		a_{k_{13},s_{1,4}}a_{k_{17},s_{1,3}}a_{k_8,s_{1,2}}a_{k_{16},s_{1,3}}a_{k_{23},s_{1,4}}
		a_{k_{24},s_{1,3}}a_{k_{15},s_{1,4}}a_{k_7,s_{1,2}}a_{k_3,s_{1,3}}=\\
		a_{k_{15},s_{1,2}}a_{k_{21},s_{1,3}}
		a_{k_{13},s_{1,4}}a_{k_{16},s_{1,3}}a_{k_{23},s_{1,4}}
		a_{k_{24},s_{1,3}}
		\end{array} $\\
		
		$\begin{array}{lllllll}
		\tau(k_{15}(s_{1,2}s_{1,4})^4k_{15}^{-1})=\tau(s_{1,3}s_{1,2}s_{1,4}s_{1,2}s_{1,4}s_{1,2}s_{1,4}s_{1,2}s_{1,4}s_{1,2}s_{1,4}s_{1,4}s_{1,2}s_{1,3})\\
		=a_{k_1,s_{1,3}}a_{k_3,s_{1,2}}a_{k_7,s_{1,4}}a_{k_{15},s_{1,2}}a_{k_{21},s_{1,4}}a_{k_{10},s_{1,2}}a_{k_{20},s_{1,4}}
		a_{k_{16},s_{1,2}}a_{k_8,s_{1,4}}a_{k_3,s_{1,2}}a_{k_7,s_{1,4}}a_{k_{15},s_{1,4}}a_{k_7,s_{1,2}}a_{k_3,s_{1,3}}=\\
		a_{k_{15},s_{1,2}}
		a_{k_{20},s_{1,4}}
		\end{array} $\\
		
		%$\begin{array}{lllllll}
		%\tau(k_{16}s_{1,2}^2k_{16}^{-1})=\tau(s_{1,3}s_{1,4}s_{1,2}s_{1,2}s_{1,2}s_{1,2}s_{1,4}s_{1,3})=
		%a_{k_1,s_{1,3}}a_{k_3,s_{1,4}}a_{k_8,s_{1,2}}a_{k_{16},s_{1,2}}a_{k_8,s_{1,2}}a_{k_{16},s_{1,2}}a_{k_8,s_{1,4}}a_{k_3,s_{1,3}}=e
		%\end{array} $\\
		$\begin{array}{llllllll}
		\tau(k_{16}s_{1,3}^2k_{16}^{-1})=\tau(s_{1,3}s_{1,4}s_{1,2}s_{1,3}s_{1,3}s_{1,2}s_{1,4}s_{1,3})
		=a_{k_1,s_{1,3}}a_{k_3,s_{1,4}}a_{k_8,s_{1,2}}a_{k_{16},s_{1,3}}a_{k_{23},s_{1,3}}a_{k_{16},s_{1,2}}a_{k_8,s_{1,4}}\\
		a_{k_3,s_{1,3}}=
		a_{k_{16},s_{1,3}}a_{k_{23},s_{1,3}}
		\end{array} $\\
		
		$\begin{array}{lllllllll}
		\tau(k_{16}s_{1,4}^2k_{16}^{-1})=\tau(s_{1,3}s_{1,4}s_{1,2}s_{1,4}s_{1,4}s_{1,2}s_{1,4}s_{1,3})=
		a_{k_1,s_{1,3}}a_{k_3,s_{1,4}}a_{k_8,s_{1,2}}a_{k_{16},s_{1,4}}a_{k_{20},s_{1,4}}a_{k_{16},s_{1,2}}a_{k_8,s_{1,2}}\\
		a_{k_3,s_{1,3}}=
		a_{k_{16},s_{1,4}}a_{k_{20},s_{1,4}}
		\end{array} $\\
		
		$\begin{array}{llllllll}
		\tau(k_{16}(s_{1,2}s_{1,3}s_{1,4}s_{1,3})^2 k_{16}^{-1})=\tau(s_{1,3}s_{1,4}s_{1,2}s_{1,2}s_{1,3}s_{1,4}s_{1,3}s_{1,2}s_{1,3}s_{1,4}s_{1,3}s_{1,2}s_{1,4}s_{1,3})\\
		=a_{k_1,s_{1,3}}a_{k_3,s_{1,4}}a_{k_8,s_{1,2}}a_{k_{16},s_{1,2}}
		a_{k_8,s_{1,3}}a_{k_{17},s_{1,4}}a_{k_{13},s_{1,3}}a_{k_{21},s_{1,2}}a_{k_{15},s_{1,3}}a_{k_{24},s_{1,4}}
		a_{k_{23},s_{1,3}}a_{k_{16},s_{1,2}}a_{k_8,s_{1,4}}a_{k_3,s_{1,3}}=\\
		a_{k_{17},s_{1,4}}a_{k_{13},s_{1,3}}a_{k_{21},s_{1,2}}a_{k_{15},s_{1,3}}a_{k_{24},s_{1,4}}
		a_{k_{23},s_{1,3}}
		\end{array} $\\
		
		$\begin{array}{llllllll}
		\tau(k_{16}(s_{1,2}s_{1,4})^4k_{16}^{-1})=\tau(s_{1,3}s_{1,4}s_{1,2}s_{1,2}s_{1,4}s_{1,2}s_{1,4}s_{1,2}s_{1,4}s_{1,2}s_{1,4}s_{1,2}s_{1,4}s_{1,3})\\
		=a_{k_1,s_{1,3}}a_{k_3,s_{1,4}}a_{k_8,s_{1,2}}a_{k_{16},s_{1,2}}a_{k_8,s_{1,4}}a_{k_3,s_{,2}}
		a_{k_7,s_{1,4}}a_{k_{15},s_{1,2}}a_{k_{21},s_{1,4}}a_{k_{10},s_{1,2}}a_{k_{20},s_{1,4}}
		a_{k_{16},s_{1,2}}a_{k_8,s_{1,4}}a_{k_3,s_{1,3}}=\\
		a_{k_{15},s_{1,2}}a_{k_{20},s_{1,4}}
		\end{array} $\\
		
		$\begin{array}{lllllll}
		\tau(k_{17}s_{1,2}^2k_{17}^{-1})=\tau(s_{1,3}s_{1,4}s_{1,3}s_{1,2}s_{1,2}s_{1,3}s_{1,4}s_{1,3})
		=a_{k_1,s_{1,3}}a_{k_3,s_{1,4}}a_{k_8,s_{1,3}}a_{k_{17},s_{1,2}}a_{k_{19},s_{1,2}}a_{k_{17},s_{1,3}}a_{k_8,s_{1,4}}\\
		a_{k_3,s_{1,3}}=
		a_{k_{17},s_{1,2}}a_{k_{19},s_{1,2}}
		\end{array} $\\
		%$\begin{array}{lllllllll}
		%\tau(k_{17}s_{1,3}^2k_{17}^{-1})=\tau(s_{1,3}s_{1,4}s_{1,3}s_{1,3}s_{1,3}s_{1,3}s_{1,4}s_{1,3})
		%=a_{k_1,s_{1,3}}a_{k_3,s_{1,4}}a_{k_8,s_{1,3}}a_{k_{17},s_{1,3}}a_{k_8,s_{1,3}}a_{k_{17},s_{1,3}}a_{k_8,s_{1,4}}a_{k_{3},s_{1,3}}=
		%e
		%\end{array} $\\
		
		$\begin{array}{llllll}
		\tau(k_{17}s_{1,4}^2k_{17}^{-1})=\tau(s_{1,3}s_{1,4}s_{1,3}s_{1,3}s_{1,3}s_{1,3}s_{1,4}s_{1,3})=
		a_{k_1,s_{1,3}}a_{k_3,s_{1,4}}a_{k_8,s_{1,3}}a_{k_{17},s_{1,4}}a_{k_{13},s_{1,4}}a_{k_{17},s_{1,3}}a_{k_8,s_{1,4}}\\
		a_{k_3,s_{1,3}}=
		a_{k_{17},s_{1,4}}a_{k_{13},s_{1,4}}
		\end{array} $\\
		
		$\begin{array}{llllll}
		\tau(k_{17}(s_{1,2}s_{1,3}s_{1,4}s_{1,3})^2 k_{17}^{-1})=\tau(s_{1,3}s_{1,4}s_{1,3}s_{1,2}s_{1,3}s_{1,4}s_{1,3}s_{1,2}s_{1,3}s_{1,4}s_{1,3}s_{1,3}s_{1,4}s_{1,3})\\
		=a_{k_1,s_{1,3}}a_{k_3,s_{1,4}}a_{k_8,s_{1,3}}a_{k_{17},s_{1,2}}a_{k_{19},s_{1,3}}
		a_{k_{12},s_{1,4}}a_{k_5,s_{1,3}}a_{k_2,s_{1,2}}a_{k_1,s_{1,3}}a_{k_3,s_{1,4}}
		a_{k_8,s_{1,3}}a_{k_{17},s_{1,3}}a_{k_8,s_{1,4}}a_{k_3,s_{1,3}}=\\
		a_{k_{17},s_{1,2}}a_{k_{19},s_{1,3}}
		\end{array} $\\
		
		$\begin{array}{llllllll}
		\tau(k_{17}(s_{1,2}s_{1,4})^4k_{17}^{-1})=\tau(s_{1,3}s_{1,4}s_{1,3}s_{1,2}s_{1,4}s_{1,2}s_{1,4}s_{1,2}s_{1,4}s_{1,2}s_{1,4}s_{1,3}s_{1,4}s_{1,3})\\
		=a_{k_1,s_{1,3}}a_{k_3,s_{1,4}}a_{k_8,s_{1,3}}a_{k_{17},s_{1,2}}
		a_{k_{19},s_{1,4}}a_{k_9,s_{1,2}}a_{k_4,s_{1,4}}a_{k_1,s_{1,2}}a_{k_2,s_{1,4}}a_{k_6,s_{1,2}}a_{k_{13},s_{1,4}}
		a_{k_{17},s_{1,3}}a_{k_8,s_{1,4}}a_{k_3,s_{1,3}}=\\
		a_{k_{17},s_{1,2}} a_{k_{13},s_{1,4}}
		\end{array} $\\
		
		$\begin{array}{llllllll}
		\tau(k_{18}s_{1,2}^2k_{18}^{-1})=\tau(s_{1,4}s_{1,2}s_{1,3}s_{1,2}s_{1,2}s_{1,3}s_{1,2}s_{1,4})
		=a_{k_1,s_{1,4}}a_{k_4,s_{1,2}}a_{k_9,s_{1,3}}a_{k_{18},s_{1,2}}a_{k_{22},s_{1,2}}a_{k_{18},s_{1,3}}\\
		a_{k_9,s_{1,2}}a_{k_4,s_{1,4}}
		=
		a_{k_{18},s_{1,2}}a_{k_{22},s_{1,2}}
		\end{array} $\\
		%$\begin{array}{llllllll}
		%\tau(k_{18}s_{1,3}^2k_{18}^{-1})=\tau(s_{1,4}s_{1,2}s_{1,3}s_{1,3}s_{1,3}s_{1,3}s_{1,2}s_{1,4})=
		%a_{k_1,s_{1,4}}a_{k_4,s_{1,2}}a_{k_9,s_{1,3}}a_{k_{18},s_{1,3}}a_{k_{9},s_{1,3}}a_{k_{18},s_{1,3}}a_{k_9,s_{1,2}}a_{k_4,s_{1,4}}
		%=e
		%\end{array} $\\
		
		$\begin{array}{llllllll}
		\tau(k_{18}s_{1,4}^2k_{18}^{-1})=\tau(s_{1,4}s_{1,2}s_{1,3}s_{1,4}s_{1,4}s_{1,3}s_{1,2}s_{1,4})=
		a_{k_1,s_{1,4}}a_{k_4,s_{1,2}}a_{k_9,s_{1,3}}a_{k_{18},s_{1,4}}a_{k_{14},s_{1,4}}a_{k_{18},s_{1,3}}\\
		a_{k_9,s_{1,2}}a_{k_4,s_{1,4}}=
		a_{k_{18},s_{1,4}}a_{k_{14},s_{1,4}}
		\end{array} $\\
		
		$\begin{array}{llllllll}
		\tau(k_{18}(s_{1,2}s_{1,3}s_{1,4}s_{1,3})^2 k_{18}^{-1})=\tau(s_{1,4}s_{1,2}s_{1,3}s_{1,2}s_{1,3}s_{1,4}s_{1,3}s_{1,2}s_{1,3}s_{1,4}s_{1,3}s_{1,3}s_{1,2}s_{1,4})\\
		=a_{k_1,s_{1,4}}a_{k_4,s_{1,2}}a_{k_9,s_{1,3}}a_{k_{18},s_{1,2}}a_{k_{22},s_{1,3}}
		a_{k_{20},s_{1,4}}a_{k_{16},s_{1,3}}a_{k_{23},s_{1,2}}a_{k_{12},s_{1,3}}
		a_{k_{19},s_{1,4}}a_{k_9,s_{1,3}}a_{k_{18},s_{1,3}}a_{k_9,s_{1,2}}a_{k_4,s_{1,4}}=\\
		a_{k_{18},s_{1,2}}a_{k_{22},s_{1,3}}
		a_{k_{20},s_{1,4}}a_{k_{16},s_{1,3}}a_{k_{12},s_{1,3}}
		\end{array}$\\
		
		$\begin{array}{llllllllll}
		\tau(k_{18}(s_{1,2}s_{1,4})^4k_{18}^{-1})=\tau(s_{1,4}s_{1,2}s_{1,3}s_{1,2}s_{1,4}s_{1,2}s_{1,4}s_{1,2}s_{1,4}s_{1,2}s_{1,4}s_{1,3}s_{1,2}s_{1,4})\\
		=a_{k_1,s_{1,4}}a_{k_4,s_{1,2}}a_{k_9,s_{1,3}}a_{k_{18},s_{1,2}}
		a_{k_{22},s_{1,4}}a_{k_{11},s_{1,2}}a_{k_5,s_{1,4}}
		a_{k_{12},s_{1,2}}a_{k_{23},s_{1,4}}a_{k_{24},s_{1,2}}a_{k_{14},s_{1,4}}
		a_{k_{18},s_{1,3}}a_{k_9,s_{1,2}}a_{k_4,s_{1,4}}=\\
		a_{k_{18},s_{1,2}}a_{k_{23},s_{1,4}}a_{k_{14},s_{1,4}}
		\end{array}$\\
		
		$\begin{array}{llllll}
		\tau(k_{19}s_{1,2}^2k_{19}^{-1})=\tau(s_{1,4}s_{1,2}s_{1,4}s_{1,2}s_{1,2}s_{1,4}s_{1,2}s_{1,4})
		=a_{k_1,s_{1,4}}a_{k_4,s_{1,2}}a_{k_9,s_{1,4}}a_{k_{19},s_{1,2}}a_{k_{17},s_{1,2}}a_{k_{19},s_{1,4}}\\
		a_{k_9,s_{1,2}}a_{k_4,s_{1,4}}=
		a_{k_{19},s_{1,2}}a_{k_{17},s_{1,2}}
		\end{array}$\\
		
		$\begin{array}{llllllllll}
		\tau(k_{19}s_{1,3}^2k_{19}^{-1})=\tau(s_{1,4}s_{1,2}s_{1,4}s_{1,3}s_{1,3}s_{1,4}s_{1,2}s_{1,4})
		=a_{k_1,s_{1,4}}a_{k_4,s_{1,2}}a_{k_9,s_{1,4}}a_{k_{19},s_{1,3}}a_{k_{12},s_{1,3}}a_{k_{19},s_{1,4}}\\
		a_{k_9,s_{1,2}}a_{k_4,s_{1,4}}=
		a_{k_{19},s_{1,3}}a_{k_{12},s_{1,3}}
		\end{array} $\\
		%$\begin{array}{llllllll}
		%\tau(k_{19}s_{1,4}^2k_{19}^{-1})=\tau(s_{1,4}s_{1,2}s_{1,4}s_{1,4}s_{1,4}s_{1,4}s_{1,2}s_{1,4})
		%=a_{k_1,s_{1,4}}a_{k_4,s_{1,2}}a_{k_9,s_{1,4}}a_{k_{19},s_{1,4}}a_{k_9,s_{1,4}}a_{k_{19},s_{1,4}}a_{k_9,s_{1,2}}a_{k_4,s_{1,4}}=e
		%\end{array} $\\
		
		$\begin{array}{llllllll}
		\tau(k_{19}(s_{1,2}s_{1,3}s_{1,4}s_{1,3})^2 k_{19}^{-1})=\tau(s_{1,4}s_{1,2}s_{1,4}s_{1,2}s_{1,3}s_{1,4}s_{1,3}s_{1,2}s_{1,3}s_{1,4}s_{1,3}s_{1,4}s_{1,2}s_{1,4})\\
		=a_{k_1,s_{1,4}}a_{k_4,s_{1,2}}a_{k_9,s_{1,4}}a_{k_{19},s_{1,2}}
		a_{k_{17},s_{1,3}}a_{k_8,s_{1,4}}a_{k_3,s_{1,3}}a_{k_1,s_{1,2}}a_{k_2,s_{1,3}}a_{k_5,s_{1,4}}
		a_{k_{12},s_{1,3}}a_{k_{19},s_{1,4}}a_{k_9,s_{1,2}}a_{k_4,s_{1,4}}=\\
		a_{k_{19},s_{1,2}}a_{k_{12},s_{1,3}}
		\end{array} $\\
		
		$\begin{array}{lllllll}
		\tau(k_{19}(s_{1,2}s_{1,4})^4k_{19}^{-1})=\tau(s_{1,4}s_{1,2}s_{1,4}s_{1,2}s_{1,4}s_{1,2}s_{1,4}s_{1,2}s_{1,4}s_{1,2}s_{1,4}s_{1,4}s_{1,2}s_{1,4})\\
		=a_{k_1,s_{1,4}}a_{k_4,s_{1,2}}a_{k_9,s_{1,4}}a_{k_{19},s_{1,2}}
		a_{k_{17},s_{1,4}}a_{k_{13},s_{1,2}}a_{k_6,s_{1,4}}a_{k_2,s_{1,2}}
		a_{k_1,s_{1,4}}a_{k_4,s_{1,2}}a_{k_9,s_{1,4}}a_{k_{19},s_{1,4}}a_{k_9,s_{1,2}}a_{k_4,s_{1,4}}=\\
		a_{k_{19},s_{1,2}}a_{k_{17},s_{1,4}}
		\end{array} $\\
		
		%$\begin{array}{llllll}
		%\tau(k_{20}s_{1,2}^2k_{20}^{-1})=\tau(s_{1,4}s_{1,3}s_{1,2}s_{1,2}s_{1,2}s_{1,2}s_{1,3}s_{1,4})=
		%a_{k_1,s_{1,4}}a_{k_4,s_{1,3}}a_{k_{10},s_{1,2}}a_{k_{20},s_{1,2}}
		%a_{k_{10},s_{1,2}}a_{k_{20},s_{1,2}}a_{k_{10},s_{1,3}}a_{k_4,s_{1,4}}=
		%e
		%\end{array} $\\
		$\begin{array}{lllllll}
		\tau(k_{20}s_{1,3}^2k_{20}^{-1})=\tau(s_{1,4}s_{1,3}s_{1,2}s_{1,3}s_{1,3}s_{1,2}s_{1,3}s_{1,4})=
		a_{k_1,s_{1,4}}a_{k_4,s_{1,3}}a_{k_{10},s_{1,2}}a_{k_{20},s_{1,3}}\\
		a_{k_{22},s_{1,3}}a_{k_{20},s_{1,2}}a_{k_{10},s_{1,3}}a_{k_4,s_{1,4}}=
		a_{k_{20},s_{1,3}}
		a_{k_{22},s_{1,3}}
		\end{array} $\\
		
		$\begin{array}{lllllll}
		\tau(k_{20}s_{1,4}^2k_{20}^{-1})=\tau(s_{1,4}s_{1,3}s_{1,2}s_{1,4}s_{1,4}s_{1,2}s_{1,3}s_{1,4})=
		a_{k_1,s_{1,4}}a_{k_4,s_{1,3}}a_{k_{10},s_{1,2}}a_{k_{20},s_{1,4}}\\
		a_{k_{16},s_{1,4}}a_{k_{20},s_{1,2}}a_{k_{10},s_{1,3}}a_{k_4,s_{1,4}}=
		a_{k_{20},s_{1,4}}
		a_{k_{16},s_{1,4}}
		\end{array} $\\
		
		$\begin{array}{lllllll}
		\tau(k_{20}(s_{1,2}s_{1,3}s_{1,4}s_{1,3})^2 k_{20}^{-1})=\tau(s_{1,4}s_{1,3}s_{1,2}s_{1,2}s_{1,3}s_{1,4}s_{1,3}s_{1,2}s_{1,3}s_{1,4}s_{1,3}s_{1,2}s_{1,3}s_{1,4})\\
		=a_{k_1,s_{1,4}}a_{k_4,s_{1,3}}a_{k_{10},s_{1,2}}a_{k_{20},s_{1,2}}a_{k_{10},s_{1,3}}
		a_{k_4,s_{1,4}}a_{k_1,s_{1,3}}a_{k_3,s_{1,2}}a_{k_7,s_{1,3}}a_{k_{11},s_{1,4}}
		a_{k_{22},s_{1,3}}a_{k_{20},s_{1,2}}a_{k_{10},s_{1,3}}a_{k_4,s_{1,4}}=\\
		a_{k_7,s_{1,3}} a_{k_{22},s_{1,3}}
		\end{array} $\\
		
		$\begin{array}{llllllll}
		\tau(k_{20}(s_{1,2}s_{1,4})^4k_{20}^{-1})=\tau(s_{1,4}s_{1,3}s_{1,2}s_{1,2}s_{1,4}s_{1,2}s_{1,4}s_{1,2}s_{1,4}s_{1,2}s_{1,4}s_{1,2}s_{1,3}s_{1,4})\\
		=a_{k_1,s_{1,4}}a_{k_4,s_{1,3}}a_{k_{10},s_{1,2}}a_{k_{20},s_{1,2}}
		a_{k_{10},s_{1,4}}a_{k_{21},s_{1,2}}a_{k_{15},s_{1,4}}a_{k_7,s_{1,2}}
		a_{k_3,s_{1,4}}a_{k_8,s_{1,2}}a_{k_{16},s_{1,4}}
		a_{k_{20},s_{1,2}}a_{k_{10},s_{1,3}}a_{k_4,s_{1,4}}=\\
		a_{k_{21},s_{1,2}}a_{k_{16},s_{1,4}}
		\end{array} $\\
		
		$\begin{array}{llllll}
		\tau(k_{21}s_{1,2}^2k_{21}^{-1})=\tau(s_{1,4}s_{1,3}s_{1,4}s_{1,2}s_{1,2}s_{1,4}s_{1,3}s_{1,4})=
		a_{k_1,s_{1,4}}a_{k_4,s_{1,3}}a_{k_{10},s_{1,4}}a_{k_{21},s_{1,2}}
		a_{k_{15},s_{1,2}}a_{k_{21},s_{1,4}}\\ 
		a_{k_{10},s_{1,3}}a_{k_4,s_{1,4}}=
		a_{k_{21},s_{1,2}}
		a_{k_{15},s_{1,2}}
		\end{array} $\\
		
		$\begin{array}{lllll}
		\tau(k_{21}s_{1,3}^2k_{21}^{-1})=\tau(s_{1,4}s_{1,3}s_{1,4}s_{1,3}s_{1,3}s_{1,4}s_{1,3}s_{1,4})=
		a_{k_1,s_{1,4}}a_{k_4,s_{1,3}}a_{k_{10},s_{1,4}}a_{k_{21},s_{1,3}}\\
		a_{k_{13},s_{1,3}}a_{k_{21},s_{1,4}}a_{k_{10},s_{1,3}}a_{k_4,s_{1,4}}=
		a_{k_{21},s_{1,3}}
		a_{k_{13},s_{1,3}}
		\end{array} $\\
		%$\begin{array}{llllll}
		%\tau(k_{21}s_{1,4}^2k_{21}^{-1})=\tau(s_{1,4}s_{1,3}s_{1,4}s_{1,4}s_{1,4}s_{1,4}s_{1,3}s_{1,4})=
		%a_{k_1,s_{1,4}}a_{k_4,s_{1,3}}a_{k_{10},s_{1,4}}a_{k_{21},s_{1,4}}
		%a_{k_{10},s_{1,4}}a_{k_{21},s_{1,4}}a_{k_{10},s_{1,3}}a_{k_4,s_{1,4}}=e
		%\end{array} $\\
		
		$\begin{array}{lllll}
		\tau(k_{21}(s_{1,2}s_{1,3}s_{1,4}s_{1,3})^2 k_{21}^{-1})=\tau(s_{1,4}s_{1,3}s_{1,4}s_{1,2}s_{1,3}s_{1,4}s_{1,3}s_{1,2}s_{1,3}s_{1,4}s_{1,3}s_{1,4}s_{1,3}s_{1,4})\\
		=a_{k_1,s_{1,4}}a_{k_4,s_{1,3}}a_{k_{10},s_{1,4}}a_{k_{21},s_{1,2}}
		a_{k_{15},s_{1,3}}a_{k_{24},s_{1,4}}a_{k_{23},s_{1,3}}a_{k_{16},s_{1,2}}a_{k_8,s_{1,3}}a_{k_{17},s_{1,4}}
		a_{k_{13},s_{1,3}}a_{k_{21},s_{1,4}}a_{k_{10},s_{1,3}}a_{k_4,s_{1,4}}=\\%
		a_{k_{21},s_{1,2}}
		a_{k_{15},s_{1,3}}a_{k_{24},s_{1,4}}a_{k_{23},s_{1,3}}a_{k_{17},s_{1,4}}
		a_{k_{13},s_{1,3}}
		\end{array} $\\
		
		$\begin{array}{lllll}
		\tau(k_{21}(s_{1,2}s_{1,4})^4k_{21}^{-1})=\tau(s_{1,4}s_{1,3}s_{1,4}s_{1,2}s_{1,4}s_{1,2}s_{1,4}s_{1,2}s_{1,4}s_{1,2}s_{1,4}s_{1,4}s_{1,3}s_{1,4})\\
		=a_{k_1,s_{1,4}}a_{k_4,s_{1,3}}a_{k_{10},s_{1,4}}a_{k_{21},s_{1,2}}
		a_{k_{15},s_{1,4}}a_{k_7,s_{1,2}}a_{k_3,s_{1,4}}a_{k_8,s_{1,2}}
		a_{k_{16},s_{1,4}}a_{k_{20},s_{1,2}}a_{k_{10},s_{1,4}}
		a_{k_{21},s_{1,4}}a_{k_{10},s_{1,3}}a_{k_4,s_{1,4}}=\\
		a_{k_{21},s_{1,2}}a_{k_{16},s_{1,4}}
		\end{array} $\\
		
		$\begin{array}{llllll}
		\tau(k_{22}s_{1,2}^2k_{22}^{-1})=\tau(s_{1,2}s_{1,3}s_{1,2}s_{1,4}s_{1,2}s_{1,2}s_{1,4}s_{1,2}s_{1,3}s_{1,2})=
		a_{k_1,s_{1,2}}a_{k_2,s_{1,3}}a_{k_5,s_{1,2}}a_{k_{11},s_{1,4}}
		a_{k_{22},s_{1,2}}\\
		a_{k_{18},s_{1,2}}a_{k_{22},s_{1,4}}
		a_{k_{11},s_{1,2}}a_{k_5,s_{1,3}}a_{k_2,s_{1,2}}=
		a_{k_{22},s_{1,2}}a_{k_{18},s_{1,2}}
		\end{array} $\\
		
		$\begin{array}{llllll}
		\tau(k_{22}s_{1,3}^2k_{22}^{-1})=\tau(s_{1,2}s_{1,3}s_{1,2}s_{1,4}s_{1,3}s_{1,3}s_{1,4}s_{1,2}s_{1,3}s_{1,2})=
		a_{k_1,s_{1,2}}a_{k_2,s_{1,3}}a_{k_5,s_{1,2}}a_{k_{11},s_{1,4}}
		a_{k_{22},s_{1,3}}\\
		a_{k_{20},s_{1,3}}a_{k_{22},s_{1,4}}
		a_{k_{11},s_{1,2}}a_{k_5,s_{1,3}}a_{k_2,s_{1,2}}=
		a_{k_{22},s_{1,3}}a_{k_{20},s_{1,3}}
		\end{array} $\\
		%$\begin{array}{lllllll}
		%\tau(k_{22}s_{1,4}^2k_{22}^{-1})=\tau(s_{1,2}s_{1,3}s_{1,2}s_{1,4}s_{1,4}s_{1,4}s_{1,4}s_{1,2}s_{1,3}s_{1,2})=
		%a_{k_1,s_{1,2}}a_{k_2,s_{1,3}}a_{k_5,s_{1,2}}a_{k_{11},s_{1,4}}
		%a_{k_{22},s_{1,4}}a_{k_{11},s_{1,4}}a_{k_{22},s_{1,4}}
		%a_{k_{11},s_{1,2}}a_{k_5,s_{1,3}}a_{k_2,s_{1,2}}=
		%e
		%\end{array} $\\
		
		$\begin{array}{llllllll}
		\tau(k_{22}(s_{1,2}s_{1,3}s_{1,4}s_{1,3})^2 k_{22}^{-1})=\tau(s_{1,2}s_{1,3}s_{1,2}s_{1,4}s_{1,2}s_{1,3}s_{1,4}s_{1,3}s_{1,2}s_{1,3}s_{1,4}s_{1,3}s_{1,4}s_{1,2}s_{1,3}s_{1,2})\\
		= a_{k_1,s_{1,2}}a_{k_2,s_{1,3}}a_{k_5,s_{1,2}}a_{k_{11},s_{1,4}}a_{k_{22},s_{1,2}}
		a_{k_{18},s_{1,3}}a_{k_9,s_{1,4}}a_{k_{19},s_{1,3}}a_{k_{12},s_{1,2}}a_{k_{23},s_{1,3}}
		a_{k_{16},s_{1,4}}a_{k_{20},s_{1,3}}a_{k_{22},s_{1,4}}\\
		a_{k_{11},s_{1,2}}a_{k_5,s_{1,3}}a_{k_2,s_{1,2}}=
		a_{k_{22},s_{1,2}}a_{k_{19},s_{1,3}}a_{k_{23},s_{1,3}}
		a_{k_{16},s_{1,4}}a_{k_{20},s_{1,3}}
		\end{array} $\\
		
		$\begin{array}{llllll}
		\tau(k_{22}(s_{1,2}s_{1,4})^4k_{22}^{-1})=\tau(s_{1,2}s_{1,3}s_{1,2}s_{1,4}s_{1,2}s_{1,4}s_{1,2}s_{1,4}s_{1,2}s_{1,4}s_{1,2}s_{1,4}s_{1,4}s_{1,2}s_{1,3}s_{1,2})\\
		=a_{k_1,s_{1,2}}a_{k_2,s_{1,3}}a_{k_5,s_{1,2}}a_{k_{11},s_{1,4}}
		a_{k_{22},s_{1,2}}a_{k_{18},s_{1,4}}a_{k_{14},s_{1,2}}a_{k_{24},s_{1,4}}a_{k_{23},s_{1,2}}a_{k_{12},s_{1,4}}a_{k_5,s_{1,2}}
		a_{k_{11},s_{1,4}}a_{k_{22},s_{1,4}}\\
		a_{k_{11},s_{1,2}}a_{k_5,s_{1,3}}a_{k_2,s_{1,2}}=
		a_{k_{22},s_{1,2}}a_{k_{18},s_{1,4}}a_{k_{24},s_{1,4}}
		\end{array} $\\
		
		%$\begin{array}{lllllll}
		%\tau(k_{23}s_{1,2}^2k_{23}^{-1})=\tau(s_{1,2}s_{1,3}s_{1,4}s_{1,2}s_{1,2}s_{1,2}s_{1,2}s_{1,4}s_{1,3}s_{1,2})
		%=a_{k_1,s_{1,2}}a_{k_2,s_{1,3}}a_{k_5,s_{1,4}}a_{k_{12},s_{1,2}}a_{k_{23},s_{1,2}}
		%a_{k_{12},s_{1,2}}a_{k_{23},s_{1,2}}a_{k_{12},s_{1,4}}a_{k_5,s_{1,3}}a_{k_2,s_{1,2}}=
		%e
		%\end{array} $\\
		$\begin{array}{llllll}
		\tau(k_{23}s_{1,3}^2k_{23}^{-1})=\tau(s_{1,2}s_{1,3}s_{1,4}s_{1,2}s_{1,3}s_{1,3}s_{1,2}s_{1,4}s_{1,3}s_{1,2})
		=a_{k_1,s_{1,2}}a_{k_2,s_{1,3}}a_{k_5,s_{1,4}}a_{k_{12},s_{1,2}}a_{k_{23},s_{1,3}}\\
		a_{k_{16},s_{1,3}}a_{k_{23},s_{1,2}}a_{k_{12},s_{1,4}}a_{k_5,s_{1,3}}a_{k_2,s_{1,2}}=
		a_{k_{23},s_{1,3}}
		a_{k_{16},s_{1,3}}
		\end{array} $\\
		
		$\begin{array}{lllll}
		\tau(k_{23}s_{1,4}^2k_{23}^{-1})=\tau(s_{1,2}s_{1,3}s_{1,4}s_{1,2}s_{1,4}s_{1,4}s_{1,2}s_{1,4}s_{1,3}s_{1,2})=
		a_{k_1,s_{1,2}}a_{k_2,s_{1,3}}a_{k_5,s_{1,4}}a_{k_{12},s_{1,2}}a_{k_{23},s_{1,4}}\\
		a_{k_{24},s_{1,4}}
		a_{k_{23},s_{1,2}}a_{k_{12},s_{1,4}}a_{k_5,s_{1,3}}a_{k_2,s_{1,2}}=
		a_{k_{23},s_{1,4}}a_{k_{24},s_{1,4}}
		\end{array} $\\
		
		$\begin{array}{llllllll}
		\tau(k_{23}(s_{1,2}s_{1,3}s_{1,4}s_{1,3})^2 k_{23}^{-1})=\tau(s_{1,2}s_{1,3}s_{1,4}s_{1,2}s_{1,2}s_{1,3}s_{1,4}s_{1,3}s_{1,2}s_{1,3}s_{1,4}s_{1,3}s_{1,2}s_{1,4}s_{1,3}s_{1,2})\\
		=a_{k_1,s_{1,2}}a_{k_2,s_{1,3}}a_{k_5,s_{1,4}}a_{k_{12},s_{1,2}}
		a_{k_{23},s_{1,2}}a_{k_{12},s_{1,3}}a_{k_{19},s_{1,4}}a_{k_9,s_{1,3}}a_{k_{18},s_{1,2}}
		a_{k_{22},s_{1,3}}a_{k_{20},s_{1,4}}a_{k_{16},s_{1,3}}a_{k_{23},s_{1,2}}\\
		a_{k_{12},s_{1,4}}a_{k_5,s_{1,3}}a_{k_2,s_{1,2}}=
		a_{k_{12},s_{1,3}} a_{k_{18},s_{1,2}}
		a_{k_{22},s_{1,3}}a_{k_{20},s_{1,4}}a_{k_{16},s_{1,3}}
		\end{array} $\\
		
		$\begin{array}{llllllll}
		\tau(k_{23}(s_{1,2}s_{1,4})^4k_{23}^{-1})=\tau(s_{1,2}s_{1,3}s_{1,4}s_{1,2}s_{1,2}s_{1,4}s_{1,2}s_{1,4}s_{1,2}s_{1,4}s_{1,2}s_{1,4}s_{1,2}s_{1,4}s_{1,3}s_{1,2})\\
		=a_{k_1,s_{1,2}}a_{k_2,s_{1,3}}a_{k_5,s_{1,4}}a_{k_{12},s_{1,2}}
		a_{k_{23},s_{1,2}}a_{k_{12},s_{1,4}}a_{k_5,s_{1,2}}a_{k_{11}, s_{1,4}}a_{k_{22},s_{1,2}}a_{k_{18},s_{1,4}}a_{k_{14},s_{1,2}}
		a_{k_{24},s_{1,4}}a_{k_{23},s_{1,2}}\\
		a_{k_{12},s_{1,4}}a_{k_5,s_{1,3}}a_{k_2,s_{1,2}}=
		a_{k_{22},s_{1,2}}a_{k_{18},s_{1,4}}a_{k_{24},s_{1,4}}
		\end{array} $\\
		
		%$\begin{array}{llllllll}
		%\tau(k_{24}s_{1,2}^2k_{24}^{-1})=\tau(s_{1,2}s_{1,4}s_{1,3}s_{1,2}s_{1,2}s_{1,2}s_{1,2}s_{1,3}s_{1,4}s_{1,2})
		%=a_{k_1,s_{1,2}}a_{k_2,s_{1,4}}a_{k_6,s_{1,3}}a_{k_{14},s_{1,2}}a_{k_{24},s_{1,2}}
		%a_{k_{14},s_{1,2}}a_{k_{24},s_{1,2}}a_{k_{14},s_{1,3}}a_{k_6,s_{1,4}}a_{k_2,s_{1,2}}=e
		%\end{array} $\\
		$\begin{array}{lllllllll}
		\tau(k_{24}s_{1,3}^2k_{24}^{-1})=\tau(s_{1,2}s_{1,4}s_{1,3}s_{1,2}s_{1,3}s_{1,3}s_{1,2}s_{1,3}s_{1,4}s_{1,2})
		a_{k_1,s_{1,2}}a_{k_2,s_{1,4}}a_{k_6,s_{1,3}}a_{k_{14},s_{1,2}}a_{k_{24},s_{1,3}}\\
		a_{k_{15},s_{1,3}}a_{k_{24},s_{1,2}}a_{k_{14},s_{1,3}}a_{k_6,s_{1,4}}a_{k_2,s_{1,2}}=
		a_{k_{24},s_{1,3}}
		a_{k_{15},s_{1,3}}
		\end{array} $\\
		
		$\begin{array}{lllllll}
		\tau(k_{24}s_{1,4}^2k_{24}^{-1})=\tau(s_{1,2}s_{1,4}s_{1,3}s_{1,2}s_{1,3}s_{1,3}s_{1,2}s_{1,3}s_{1,4}s_{1,2})
		=a_{k_1,s_{1,2}}a_{k_2,s_{1,4}}a_{k_6,s_{1,3}}a_{k_{14},s_{1,2}}a_{k_{24},s_{1,4}}\\
		a_{k_{23},s_{1,4}}a_{k_{24},s_{1,2}}a_{k_{14},s_{1,3}}a_{k_6,s_{1,4}}a_{k_2,s_{1,2}}=
		a_{k_{24},s_{1,4}}
		a_{k_{23},s_{1,4}}
		\end{array} $\\
		
		$\begin{array}{llllllll}
		\tau(k_{24}(s_{1,2}s_{1,3}s_{1,4}s_{1,3})^2 k_{24}^{-1})=\tau(s_{1,2}s_{1,4}s_{1,3}s_{1,2}s_{1,2}s_{1,3}s_{1,4}s_{1,3}s_{1,2}s_{1,3}s_{1,4}s_{1,3}s_{1,2}s_{1,3}s_{1,4}s_{1,2})\\
		=a_{k_1,s_{1,2}}a_{k_2,s_{1,4}}a_{k_6,s_{1,3}}a_{k_{14},s_{1,2}}
		a_{k_{24},s_{1,2}}a_{k_{14},s_{1,3}}a_{k_6,s_{1,4}}a_{k_2,s_{1,3}}a_{k_5,s_{1,2}}a_{k_{11},s_{1,3}}
		a_{k_7,s_{1,4}}a_{k_{15},s_{1,3}}a_{k_{24},s_{1,2}}\\
		a_{k_{14},s_{1,3}}a_{k_6,s_{1,4}}a_{k_2,s_{1,2}}=
		a_{k_{11},s_{1,3}}a_{k_{15},s_{1,3}}
		\end{array} $\\
		
		$\begin{array}{lllllllll}
		\tau(k_{24}(s_{1,2}s_{1,4})^4k_{24}^{-1})=\tau(s_{1,2}s_{1,4}s_{1,3}s_{1,2}s_{1,2}s_{1,4}s_{1,2}s_{1,4}s_{1,2}s_{1,4}s_{1,2}s_{1,4}s_{1,2}s_{1,3}s_{1,4}s_{1,2})\\
		a_{k_1,s_{1,2}}a_{k_2,s_{1,4}}a_{k_6,s_{1,3}}a_{k_{14},s_{1,2}}a_{k_{24},s_{1,2}}
		a_{k_{14},s_{1,4}}a_{k_{18},s_{1,2}}a_{k_{22},s_{1,4}}a_{k_{11},s_{1,2}}a_{k_5,s_{1,4}}a_{k_{12},s_{1,2}}
		a_{k_{23},s_{1,4}}a_{k_{24},s_{1,2}}\\
		a_{k_{14},s_{1,3}}a_{k_6,s_{1,4}}a_{k_2,s_{1,2}}=
		a_{k_{14},s_{1,4}}a_{k_{18},s_{1,2}}a_{k_{23},s_{1,4}}
		\end{array} $\\

		This implies:\\
		$a_{k_7,s_{1,3}}=a_{k_{11},s_{1,3}}^{-1}=a_{k_{15},s_{1,3}}=a_{k_{20},s_{1,3}}=a_{k_{22},s_{1,3}}^{-1}=a_{k_{24},s_{1,3}}^{-1}$;
		$a_{k_{12},s_{1,3}}=a_{k_{13},s_{1,4}}^{-1}=a_{k_{17},s_{1,2}}=a_{k_{17},s_{1,4}}=a_{k_{19},s_{1,2}}^{-1}=a_{k_{19},s_{1,3}}^{-1}$;
		$a_{k_{13},s_{1,3}}=a_{k_{14},s_{1,4}}=a_{k_{18},s_{1,4}}^{-1}=a_{k_{21},s_{1,3}}^{-1}$;
		$a_{k_{15},s_{1,2}}=a_{k_{16},s_{1,4}}=a_{k_{20},s_{1,4}}^{-1}=a_{k_{21},s_{1,2}}^{-1}$;
		$a_{k_{16},s_{1,3}}=a_{k_{23},s_{1,3}}^{-1}$;
		$a_{k_{18},s_{1,2}}=a_{k_{22},s_{1,2}}^{-1}$;
		$a_{k_{23},s_{1,4}}=a_{k_{24},s_{1,4}}^{-1}$.\\
		
		It follows that $PJ_4$ is generated by 
		$a_{k_7,s_{1,3}},
		a_{k_{12},s_{1,3}},
		a_{k_{13},s_{1,3}},
		a_{k_{15},s_{1,2}},
		a_{k_{16},s_{1,3}},
		a_{k_{18},s_{1,2}},
		a_{k_{23},s_{1,4}}$, and a complete set of relations is:
		
		$a_{k_7,s_{1,3}}a_{k_{23},s_{1,4}}^{-1}a_{k_{16},s_{1,3}}^{-1}a_{k_{12},s_{1,3}}a_{k_{13},s_{1,3}}a_{k_{15},s_{1,2}}^{-1}=1$;
		$a_{k_{13},s_{1,3}}a_{k_{18},s_{1,2}}a_{k_{23},s_{1,4}}=1$; and
		
		$a_{k_7,s_{1,3}}a_{k_{18},s_{1,2}}^{-1}a_{k_{12},s_{1,3}}^{-1}a_{k_{16},s_{1,3}}^{-1}a_{k_{15},s_{1,2}}=1$.
		
	Therefore, using $a_{k_{23},s_{1,4}}^{-1}= a_{k_{13},s_{1,3}} a_{k_{18},s_{1,2}} $
and  $a_{k_{16},s_{1,3}}^{-1} = a_{k_{12},s_{1,3}} a_{k_{18},s_{1,2}}  a_{k_7,s_{1,3}}^{-1} a_{k_{15},s_{1,2}}^{-1}$,	we get 

\begin{align*}
PJ_4 &=\langle  \ a_{k_7,s_{1,3}},   \  a_{k_{12},s_{1,3}}, \ a_{k_{13},s_{1,3}}, \ a_{k_{15},s_{1,2}},  a_{k_{18},s_{1,2}} \ | \\
&a_{k_7,s_{1,3}} a_{k_{13},s_{1,3}} a_{k_{18},s_{1,2}} 
a_{k_{12},s_{1,3}}  a_{k_{18},s_{1,2}}a_{k_7,s_{1,3}}^{-1} a_{k_{15},s_{1,2}}^{-1} a_{k_{12},s_{1,3}}a_{k_{13},s_{1,3}}a_{k_{15},s_{1,2}}^{-1}=1
\rangle. \qedhere
\end{align*}
	\end{proof}
		
	\bibliographystyle{alpha}
	\bibliography{cactus}
\end{document}